\documentclass[11pt]{article}
  \usepackage{graphicx} 

\title{Semidiscrete Modeling of Dislocation-Disclination Systems:
Finite Element Formulation }

\author{Edoardo Fabbrini\footnote{SACRA, Graduate School of Science, Kyoto University, Japan \url{fabbrini.edoardo.2w@kyoto-u.ac.jp}}
, Pierluigi Cesana\footnote{Institute of Mathematics for Industry, Kyushu University, Japan, \url{cesana@math.kyushu-u.ac.jp}
}, 
 Andr\'{e}s A. Le\'{o}n Baldelli\footnote{
Sorbonne Université, CNRS,
Institut Jean Le Rond d’Alembert, France, \url{leon.baldelli@cnrs.fr}
}, 
Marco Morandotti\footnote{
Dipartimento di Scienze Matematiche ``G.~L.~Lagrange'', Politecnico di Torino, Italy, \url{marco.morandotti@polito.it}
}}

\date{\today}

\usepackage{amssymb}
\usepackage{amsmath}

\usepackage[utf8]{inputenc}
\usepackage[lined]{algorithm2e}
\usepackage{minitoc}
\usepackage{mathtools}
\usepackage{stmaryrd} 
\usepackage{fancyvrb}
\usepackage{color}
\usepackage{multicol}
\usepackage{datetime}
\usepackage[left=3cm,right=3cm]{geometry}
\usepackage{soul}
\usepackage{cancel}
\usepackage[english]{babel}
\usepackage{booktabs}
\usepackage{mathrsfs}
\usepackage{amsthm}
\usepackage{comment}
\usepackage{chngcntr}
\usepackage{bm}
\usepackage{enumitem} 

\usepackage[percent]{overpic}

\counterwithin{equation}{section}

\usepackage[pdftitle={},pdfsubject={},pdfkeywords={},pdfproducer={Latex with hyperref},pdfcreator={pdflatex}]{hyperref}

\theoremstyle{plain} 
\newtheorem{theorem}{Theorem}[section]
\newtheorem{proposition}[theorem]{Proposition}
\newtheorem{lemma}[theorem]{Lemma}

\theoremstyle{definition} 
\newtheorem{definition}[theorem]{Definition}

\theoremstyle{remark} 
\newtheorem{remark}[theorem]{Remark}

\renewenvironment{proof}[1][Proof]{\noindent\textbf{#1.} }{\hfill$\square$\vspace{5pt}}

\usepackage{latexsym}
\usepackage{boxedminipage}
\usepackage{listings}
\usepackage{minitoc}
\usepackage{ifpdf}
\usepackage{subcaption}
\usepackage{array}

\usepackage[T1]{fontenc}      
\usepackage{graphicx}
\graphicspath{ {.}, {./images} }
\usepackage{xcolor}
\usepackage{booktabs}
\usepackage{geometry}         
\usepackage{float}            
\usepackage{enumitem}         
\usepackage{caption}         
\usepackage{calligra}

\usepackage{tikz}
\usetikzlibrary{calc}

\newcommand{\edgin}{\mathcal{E}_{\text{in}}}
\newcommand{\edgbdouter}{\mathcal{E}_{\text{out}}}
\newcommand{\edgbdinner}[1]{\mathcal{E}_{\text{core},#1}}

\newcommand{\IP}{\alpha}

\newcommand{\cof}{\operatorname{cof}}

\newcommand{\Omegaeps}{\Omega_{\varepsilon}}
\newcommand{\Div}{\operatorname{Div}}
\newcommand{\de}{\bm{\delta}}

\newcommand{\CC}{\mathbb{C}}

\newcommand{\cG}{\mathcal{G}}
\newcommand{\cI}{\mathcal{I}}
\newcommand{\inc}{\operatorname{inc}}
\newcommand{\scrC}{\mathscr{C}}

\makeatletter
\def\@splitop#1#2\@nil{$\mathscr{#1}\!\!$\calligra#2\,\,}
\newcommand*\DeclareCursiveOperator[2]
{
\newcommand#1{\mathop{\mbox{\@splitop#2\@nil}}\nolimits}
}
\DeclareCursiveOperator{\Xv}{C}
\makeatother

\newcommand{\ud}{\mathrm{d}}

\newcommand{\sym}{\mathrm{sym}}
\newcommand{\R}{\mathbb{R}}

\newcommand{\Huno}{\mathcal H^1}

\newcommand{\average}{{\mathchoice {\kern1ex\vcenter{\hrule
height.4pt width 8pt depth0pt}
\kern-11pt} {\kern1ex\vcenter{\hrule height.4pt width 4.3pt
depth0pt} \kern-7pt} {} {} }}
\newcommand{\ave}{\average\int} 

\newcommand{\ep}{\varepsilon}
\newcommand{\ce}{\ep}

\begin{document}

\maketitle

\begin{abstract}
We present a numerical formulation for the resolution of finite systems of interacting edge dislocations and wedge disclinations. 
The approach is based on a finite element discretization of a fourth-order elliptic boundary value problem arising from the mechanical equilibrium equations of plane-strain linear elasticity in the presence of kinematic incompatibilities.
The numerical implementation follows a continuous interior penalty discontinuous Galerkin framework and relies on the solution of a finite set of local cell problems.
Our  method is validated against analytical benchmarks and used to explore several 
non-trivial dislocation-disclination configurations, illustrating how translational and rotational incompatibilities shape the stress state of the body.

\end{abstract}


\tableofcontents

\section{Introduction}

The modeling of systems of dislocations and disclinations is essential for understanding and predicting the large scale behavior of metal alloys, elastic crystals, and crystalline membranes. This task is challenging due to the wide range of length scales involved and the complex nature and magnitude of the associated mechanical fields.
Dislocations are \textit{translational} topological defects, first introduced by Volterra \cite{V07}. In continuum theories \cite{Z97, RV92}, they arise as distributional solutions of mechanical equilibrium with incompatible kinematics and are associated with infinite elastic energy.
By contrast, disclinations, also introduced by Volterra, represent \textit{rotational} mismatches at the level of the crystal lattice.
They are not strictly topological defects since they have finite energy, although they still produce singular mechanical stresses.

Interactions between disclinations and dislocations are well documented experimentally and have a strong influence on material behavior. 
In crystal plasticity, for example, their combined effects are important  in the modeling of kinking  \cite{HAGIHARA10, I19}, grain boundaries \cite{Gertsman89,LI72}, crack-tip plasticity
\cite{Xu95,Xu97}. 
Due to their high energy, which scales with the square of the radius of their domain, isolated disclinations are rarely found in nature. 
Instead, they tend to occur in pairs with opposite mismatch angles, a configuration that is energetically favorable due to the screening of mechanical stresses \cite{RRK2018, CPL14}. 
In graphene, configurations such as disclination dipoles, Stone--Wales defects, and other complex configurations significantly influence electrical, magnetic, and mechanical properties, as well as chemical reactivity \cite{KVV2016}.
Recent progress in the modeling of disclinations and dislocations has been driven by the development of the 
g.disclination theory, which provides a unified continuum framework capable of incorporating phase transformations, grain boundaries, and a wide range of plasticity mechanisms. For further details, we refer the reader to \cite{acharya15, ZA2018, ZHANG18, FRESSENGEAS2020104092}.

The aim of this paper is to introduce a numerical method for solving finite systems of edge dislocations and wedge disclinations
within a unified variational framework. 
Our approach leverages the variational structure of the underlying mechanical model and is based on the solution of cell problems for the Airy stress potential. The numerical implementation employs the finite element method in combination with a continuous interior penalty discontinuous Galerkin formulation  for a fourth-order elliptic problem.
From a mechanical standpoint, we assume linearized kinematics in isotropic elasticity under the planar strain regime. Within this framework, we model Volterra disclinations and dislocations as solutions to the mechanical equilibrium problem, which accounts for kinematic incompatibilities in traction-free conditions. In continuum mechanics, both dislocations and disclinations are characterized by divergent stresses. Singularities in the stress field are regularized using the so-called \textit{core radius} method.
This method involves introducing a length scale, denoted as $\varepsilon>0$, representing the size of the region around the defect where stresses become singular, and the continuum model breaks down. 
This establishes the \textit{semidiscrete} nature of the model.
This approach has been successfully analyzed for dislocations by Cermelli and Leoni \cite{CermelliLeoni06}, building upon the work of Bethuel, Brezis, and Hélein on Ginzburg-Landau vortices (see also \cite{GarroniLeoniPonsiglione10,DeLucaGarroniPonsiglione12,BlassMorandotti17,Ginster19_2,AlicandroDeLucaPalombaroPonsiglione2025}), and later extended to disclinations \cite{Cesana2024a}. Experimental estimations of these regions suggest that their size is of the order of a few nanometers, as shown in \cite{Peierls40}.

The main contributions of this work are as follows.

\begin{enumerate}[leftmargin=12pt, topsep=0pt, itemsep=0pt, partopsep=0pt, parsep=0pt]
    \item Our method simultaneously handles any finite combination of Volterra edge dislocations and wedge disclinations in general domains. While there exists a large body of research on phenomenological models for microplasticity in both infinitesimal and finite elasticity, much of the existing literature focuses on empirical descriptions of asymmetries and mismatches via phase variables, based on the   phase field approach \cite{RODNEY03, WANG01, WANG10, cai06, Kundin11}. These studies primarily aim to model statistically stored dislocations and defects, which are relevant to understanding plastic deformation processes at low to moderate strains, particularly in materials undergoing steady-state plastic flow 
\cite{FleckMullerAshbyHutchinson94}.
In contrast, our approach focuses on the fundamental question of modeling interacting systems of Volterra dislocations and disclinations, which represent idealized, isolated defects purely defined by geometry.
We explore a variety of configurations where these defects interact in non-trivial ways, investigating constructive/destructive interactions.
 
    \item  Our method is based on a variational principle, providing a full variational characterization for the mechanical system with incompatible elasticity, formulated in terms of the Airy stress function. This function serves as the global unknown of the system, simultaneously accounting for traction-free boundary conditions and kinematic incompatibility. This approach offers advantages over the standard implementations of the Airy stress function method in the engineering literature, which typically involves a two-step process: first solving for an infinite-domain solution, then adding  terms that enforce zero-stress boundary conditions \cite{Becker21,Zdzisaw99}.
    
Specifically, we solve for the Airy potential, which reduces a tensorial problem (involving stresses and strains) to a scalar problem, albeit of fourth order. However, a complication arises when formulating the boundary conditions for the Airy potential. Traction-free conditions on the stress at the boundary translate, formally, into tangential conditions on the hessian of the Airy potential \cite{Michell, Sadd25}. 
We rely on a characterization developed in \cite{Cesana2024a} to show that these non-classical tangential conditions are equivalent to parametric Dirichlet-type conditions for the Airy potential. The full variational characterization of the mathematical model, presented in \cite{CFM2025} for smooth domains, is generalized here for a class of Lipschitz-type domains, which are well-suited for finite element problems.

\item Our numerical formulation is highly general. Once the geometry of the problem is defined, it proceeds via the computation of cell formulas. After the cell formulas are obtained, any configuration of edge dislocations or wedge disclinations can be efficiently generated at the postprocessing stage, without the need to rerun the finite element problem. This approach leads to significant time savings. 
More precisely, our approach involves solving \(3N\) cell formula problems, where \(N\) is the total number of defects (in the core radius regularization this refers to the total number of such cores). Once the corresponding fundamental solutions are determined, any configuration of defects for a given geometry, represented by \(K\) disclinations and \(J\) dislocations (under the condition \(N = K + J\)), can be computed on the fly during postprocessing. 

 \end{enumerate}

\smallskip

The outline of the paper is as follows. After presenting the mathematical background of the semidiscrete modeling approach for systems of disclinations and dislocations, in Section~\ref{2511271518} we present our main analytical result, Theorem~\ref{2502282100}, which establishes the well posedness of mechanical equilibrium problems in the presence of kinematic incompatibility.  
The theorem is formulated both in the standard stress/strain framework and in terms of the Airy stress potential. Additionally, it provides a characterization of the equilibrium solution in terms of cell formulas for the Airy potential, which serves as the basis for the numerical implementation.
In Section \ref{sec:202509161701}, we describe our numerical implementation, which   leverages the well-posedness theorem, with particular emphasis on the characterization of the mechanical equilibrium problem via cell formulas.
In Section \ref{2512112300}, we present simulations of relevant configurations involving disclination-dislocation interacting systems. Additionally, we provide supplementary materials, including video collections of parametric simulations.
  
\section{Preliminaries and analytical model}\label{2511271518}

We denote by $\Omega$ a 
domain in $\mathbb{R}^2$ with Lipschitz boundary (the precise requirements on $\Omega$ will be specified in Definition~\ref{def_Omega} below). 
We consider a collection of $J\in\mathbb{N}$ edge dislocations placed at $x^{(1)},\ldots,x^{(J)}\in\Omega$ and characterized by their Burgers vectors $b^{1},\ldots,b^{J}\in\mathbb{R}^2$, and a collection of $K\in\mathbb{N}$ wedge disclinations placed at $y^{(1)},\ldots,y^{(K)}\in\Omega$ and characterized by their Frank angles $s^{1},\ldots,s^{K}\in\mathbb{R}$. 
We require that $\{\xi^{(i)}\}_{i=1}^{N}=\{x^{(j)}\}_{j=1}^{J}\cup\{y^{(k)}\}_{k=1}^{K}$ results in a set of $N=J+K$ distinct points (so that no two different defects are occupying the same position, \emph{i.e.}, $\xi^{(i_1)}\neq \xi^{(i_2)}$ if $i_1\neq i_2$).

We work within the framework of linearized elasticity 
in planar strain regime \cite{ciarlet97}. The constitutive law relating the strain tensor $\epsilon\in \R^{2\times2}_{\sym}$ to the stress tensor $\sigma\in\R^{2\times2}_{\sym}$ is given by $\sigma = \mathbb{C}\epsilon$, where $\mathbb{C}$ denotes the Cauchy elasticity tensor. For isotropic materials, this relation can be inverted and takes the form
\begin{equation}\label{strain_stress}
\epsilon_{11}=\frac{1+\nu}{E}\Big((1-\nu)\sigma_{11}-\nu\sigma_{22}\Big)\,,\quad \epsilon_{12}=\frac{1+\nu}{E}\sigma_{12}\,,\quad \epsilon_{22}=\frac{1+\nu}{E}\Big((1-\nu)\sigma_{22}-\nu\sigma_{11}\Big)\,.
\end{equation}
We describe incompatible kinematics by means of the incompatibility operator $\mathrm{inc}\colon H^k(\Omega;\mathbb{R}^{2\times 2}_{\sym})\to H^{k-2}(\Omega)$ (see 
\cite{Angoshtari2016, Yavari2013, Yavari2020,  vG2017, Acharya99a, ACGK})
defined by
\begin{equation}\label{inc_operator}
\mathrm{inc}\,\epsilon \coloneqq \partial^2_{x_2^2}\epsilon_{11}-2\partial^2_{x_1x_2}\epsilon_{12}+\partial^2_{x_1^2}\epsilon_{22}\,,
\end{equation}
where differentiation is intended in the weak sense.
In the presence of dislocations and disclinations, the condition of kinematic incompatibility reads \cite{SN88,vaGoethemDupret2012,Cesana2024a,CFM2025}
\begin{equation}\label{inc_eps}
\mathrm{inc}\,\epsilon=\sum_{j=1}^J |b^{j}|\partial_{\frac{(b^{j})^\perp}{|b^{j}|}} \de_{x^{(j)}}-\sum_{k=1}^K s^{k}\de_{y^{(k)}} \eqqcolon \zeta\,,
\end{equation}
where $\de_{\xi^{(i)}}$ is the Dirac delta supported at $\xi^{(i)}\in\Omega$ and the operator $\partial_{(b^j)^\perp/|b^j|}$ is the derivative in the direction of the (unit) vector $(b^j)^\perp/|b^j|$ and is understood in the sense of distributions.
The measure $\zeta$ is a measure of the presence of defects in the body; in the absence of dislocations and disclinations (\emph{i.e.}, when 
$\zeta=0$), the body is in the condition of geometric compatibility and $\mathrm{inc}\,\epsilon=0$ in $\Omega$.

When looking for equilibrium configurations of an elastic  body, 
one seeks to solve the  minimization problem
\begin{equation}\label{min_prob_eps}
\min \big\{\mathcal{W}(\epsilon;\Omega):\epsilon\in L^2(\Omega;\mathbb{R}^{2\times 2}_{\sym}), \text{ \eqref{inc_eps} is satisfied}\big\},
\end{equation}
where $\mathcal{W}(\cdot;\Omega)\colon L^2(\Omega;\mathbb{R}^{2\times 2}_{\sym})\to\mathbb{R}$ is the mechanical energy functional
\begin{equation}\label{energy_E}
\mathcal{W}(\epsilon;\Omega)\coloneqq \frac{1}{2}\int_\Omega \sigma:\epsilon\,\ud x = \frac12\int_\Omega \CC\epsilon:\epsilon\,\ud x.
\end{equation}
The associated Euler-Lagrange equation is
\begin{equation}\label{EL_eq_strain}
\begin{cases}
\Div\sigma=0 & \text{in $\Omega$,} \\
\sigma\, n =0 & \text{on $\partial\Omega$,} \\
\inc \epsilon = \zeta 
& \text{in $\Omega$,}
\end{cases}
\end{equation}
where $\Div$ is the row-wise divergence operator, and the traction-free  condition emerges as the \emph{natural boundary condition}.
The Airy stress function method introduces a scalar potential 
$v\colon\Omega\to\mathbb{R}$ such that the stress can be written as (see \cite[section 5.7]{ciarlet97}) 
\begin{subequations}\label{eq_Airyoperatorgen}
\begin{equation}\label{eq_Airyoperator1}
\sigma=\sigma[v]=\mathcal{A}(v),
\end{equation}
where $\mathcal{A}\colon H^{k+2}(\Omega)\to H^{k}(\Omega;\mathbb{R}^{2\times2}_{\sym})$ is defined by
\begin{equation}\label{eq_Airyoperator}
\mathcal{A}(v)\coloneqq \cof(\nabla^2 v)=\begin{pmatrix}
v_{x_2x_2} & -v_{x_1x_2} \\
-v_{x_2x_1} & v_{x_1x_1}
\end{pmatrix}
\end{equation}
\end{subequations}
and the cofactor operator $\cof\colon\R^{2\times2}\to\R^{2\times2}$ acts on $2\times2$ matrices in the following way:
$$\begin{pmatrix}
m_{11} & m_{12} \\
m_{21} & m_{22}
\end{pmatrix}=m\mapsto \cof(m)=
\begin{pmatrix}
m_{22} & -m_{21} \\
-m_{12} & m_{11}
\end{pmatrix}.$$
A straightforward computation shows that, if $\sigma$ and $v$ are related by \eqref{eq_Airyoperatorgen}, then $\Div\sigma[v]=0$ is automatically satisfied; moreover, owing to \eqref{strain_stress} and \eqref{eq_Airyoperatorgen}, the incompatibility operator \eqref{inc_operator} reads $\mathrm{inc}\,\epsilon = \frac{1-\nu^2}{E}\Delta^2 v$,
so that the kinematic incompatibility condition \eqref{inc_eps} becomes, for the Airy stress function $v\in H^{k+2}(\Omega)$,
\begin{equation}\label{inc_v}
\frac{1-\nu^2}{E}\Delta^2 v= \zeta. 
\end{equation}
By \eqref{strain_stress} and \eqref{eq_Airyoperatorgen}, the mechanical energy functional $\mathcal{W}$ in \eqref{energy_E} depending on the strain $\epsilon$ can be rewritten as a functional $\mathcal{G}(\cdot;\Omega) \colon H^{2}(\Omega)\to \mathbb{R}$ defined in the following way:
\begin{equation}\label{eq_energyv}
\cG(v;\Omega) \coloneqq  \frac12\frac{1+\nu}{E} \int_\Omega \big[|\nabla^2 v|^2-\nu(\Delta v)^2\big]\,\ud x.
\end{equation}
The aim, now, is to see equation \eqref{inc_v} as the Euler--Lagrange equation of a suitable functional, that keeps the presence of the defects into account through the measure $\zeta$.
As proposed in \cite{Cesana2024a}, the functional $\mathcal{I}(\cdot;\Omega) \colon H^2(\Omega)\to\mathbb{R}$ defined by
\begin{equation}\label{eq_functional_I}
\mathcal{I}(v;\Omega) \coloneqq \mathcal{G}(v;\Omega)+\langle\zeta,v\rangle,
\end{equation}
where $\langle \zeta,v\rangle$ is the duality pairing between the measure $\zeta$ and the function $v$, yields \eqref{inc_v} as its Euler--Lagrange equation, together with the boundary condition $\nabla^2 v\,t=0$ on $\partial \Omega$, which is the translation, in the Airy variable, of the condition $\sigma\,n=0$ on $\partial\Omega$, as a simple computation using \eqref{eq_Airyoperatorgen} shows.

We adopt the \emph{core-radius approach} (see, \emph{e.g.}, \cite{CermelliLeoni06}) and we formulate the problem by removing small disks around each defect $\xi^{(i)}$\,, ($i=1,\ldots,N$) and by transferring the effects of $\de_{\xi^{(i)}}$ to the boundaries of these disks.
More precisely, let $\varepsilon>0$ be such that the open disks $B_\varepsilon^i\coloneqq B_\varepsilon(\xi^{(i)})$, for $i=1,\ldots,N$, are all disjoint and well contained in~$\Omega$ (this is possible since the points $\xi^{(i)}$'s are all distinct), and define
\begin{equation}\label{eq_Omega_eps}
\Omega_\varepsilon\coloneqq\Omega\setminus\bigg(\bigcup_{i=1}^N\overline{B}_\varepsilon^i\bigg),\quad\text{with}\quad \partial \Omega_\varepsilon = \partial \Omega\cup\bigg(\bigcup_{i=1}^N \partial B_\varepsilon^i\bigg),
\end{equation}
and the $\varepsilon$-regularized functional
\begin{equation}\label{eq_energyperforated}
\cI_{\varepsilon}(v;\Omega_{\varepsilon})\coloneqq \,
\cG(v;\Omega_{\varepsilon})    +\sum_{j=1}^J\frac{1}{2\pi\ce}\int_{\partial B_{\ce}^j} \langle\nabla v, \Pi(b^{j})\rangle\,\ud\Huno
   +\sum_{k=1}^K \frac{s^{k}}{2\pi\ce}\int_{\partial B_\ce^k} v\,\ud\Huno\,,
 \end{equation}
where the boundary integrals are the approximation of the corresponding atomic measures in $\zeta$ (see \eqref{inc_eps}).
We seek to minimize the functional in \eqref{eq_energyperforated} in the class
\begin{equation}\label{eq_competitorsperforatedfloat}
\scrC(\Omega_\varepsilon):= 
\big\{v\in H^2_0(\Omega): \text{$v=a^i$ in $B_\ce^i$\,, for some affine functions $a^i$, $i=1,\dots,N$}\}\,.
\end{equation}
The functional in \eqref{eq_energyperforated} was introduced in \cite{CFM2025} to model an elastic body with kinematic incompatibilities of translational type (measured by the Burgers vectors \( b^{j} \)) and rotational type (measured by the Frank angles \( s^{k} \)). Owing to the adopted regularization strategy, the corresponding equilibrium configurations exhibit smooth stress fields.
Formally, as \( \varepsilon \to 0 \), the functional \( \cI_{\varepsilon}(\cdot; \Omega_\varepsilon) \) converges to the limit functional $\cI(\cdot;\Omega)$ in \eqref{eq_functional_I}. The variational convergence of this family has been partially investigated in \cite{Cesana2024a}.

The class  functions \eqref{eq_competitorsperforatedfloat} provides the appropriate framework for identifying mechanical equilibrium solutions. It was first introduced in \cite{Cesana2024a} and encodes unknown affine boundary conditions on the internal boundaries.  
It was proved in \cite[Proposition~A.2]{Cesana2024a} that these boundary conditions characterize the fact that \begin{equation}\label{eq_hesstang0}
\nabla^2 v\,t=0\qquad \text{on $\partial\Omega_\varepsilon$\,,} 
\end{equation}
provided that both $v$ and $\partial\Omega$ is of class $C^2$.
Owing to \eqref{eq_Airyoperatorgen}, this condition ensures zero normal stress on $\partial\Omega_\varepsilon$\,, and was first recognized by Michell \cite{Michell}.
In what follows, we extend the variational theory developed in \cite{CFM2025} for smooth domains to a broader class, defined below, that includes polygonal domains.

\begin{table}[h!]
    \centering
    \begin{tabular}{ll|ll}
        \toprule
        \textbf{Symbol} & \textbf{Description} & \textbf{Symbol} & \textbf{Description} \\
        \midrule      
        $\Omega$ & Reference configuration  & $\Omega_{\varepsilon}$  & Reference configuration  \\
                &  (simply connected) &   &  (non-simply connected) \\
        $\varepsilon>0$ & Core radius & $x = (x_1, x_2)$ & Point in $\Omega,\Omega_{\varepsilon}$ \\
        $\de_x$ & Dirac delta function at $x$ & $v$ & Airy stress function \\
        $\epsilon$; $\epsilon_{ij}$ & Strain tensor; components & $\sigma$; $\sigma_{ij}$ & Stress tensor; components \\
        $s,s^{k} \in \mathbb{R}$ & Frank angle & $b,b^{j} \in \mathbb{R}^2$ & Burgers vector \\
        $y^{(k)}$ & Position of disclination &  $x^{(j)}$ & Position of dislocation\\
        $x^\perp\coloneqq (-x_2,x_1)$ & $\pi/2$-rotation of a vector &  $\Pi x\coloneqq -x^\perp$ & $-\pi/2$-rotation of a vector\\
        $n=(n_1,n_2)$ & outward normal to a domain & $t=n^\perp$ & unit tangent to a domain\\ 
        $K \in \mathbb{N}$ & Number of wedge disclinations & $J \in \mathbb{N}$ & Number of edge dislocations \\
        $E>0$ & Young modulus & $\nu \in (-1,\frac{1}{2})$ & Poisson ratio \\
        \bottomrule
    \end{tabular}
    \caption{Relevant material and geometric parameters.}
    \label{table:2024}
\end{table}

\begin{table}[h!]
    \centering
    \begin{tabular}{llc}
        \toprule
        \textbf{Symbol} & \textbf{Description}  \\
        \midrule
         $C(\Omega)$ & Space of continuous functions defined in $\Omega$  \\    $C^k(\Omega)$ & Space of $k$-time differentiable functions defined in $\Omega$, $k \ge 1$ \\     
         $L^p(\Omega)$ & Space of $p$-th power Lebesgue integrable functions defined in $\Omega$, $p \ge 1$ \\
         $W^{k, p}(\Omega)$ & $L^p$-maps having $k$-th order derivatives with components in $L^p(\Omega)$, $k \in \mathbb{N}$ \\
         $H^{2}(\Omega)$ & Special notation for $W^{2, 2}(\Omega)$ \\
        $H_0^{2}(\Omega)$ & Set of functions $f \in H^2(\Omega)$ such that $f = \partial_n f = 0$ on $\partial \Omega$  \\
        $H^{-2}$ & Dual space of $H^2_0(\Omega)$ \\
        $\nabla f$, $\nabla^2 f$ & Gradient and Hessian matrix of $f$, respectively \\
        $\Delta f$, $\Delta^2 f$ &  Laplacian and bilaplacian of $f$, respectively ($\Delta^2 f:=\Delta\Delta f$) \\
         \bottomrule
    \end{tabular}
    \caption{Notation for function spaces and operators.}
    \label{table:202505131228}
\end{table}

\begin{definition}[Piecewise-$C^2$ Lipschitz domains]\label{def_Omega}
Let $\Omega\subset\mathbb{R}^2$ be a bounded, open, and simply connected set with Lipschitz boundary $\Gamma \coloneqq \partial\Omega$. 
We say that $\Omega$ is a \emph{piecewise-$C^2$  Lipschitz domain}  if $\Gamma$ is globally Lipschitz and made of the union of $L\geq1$ 
portions $\{\Gamma^\ell\}_{\ell=1}^{L}$ of class $C^2$.
The endpoints of each portion $\Gamma^\ell$ of the boundary are named in the following way:
$$\Gamma^\ell=[P^\ell,P^{\ell+1}],\qquad\text{with}\quad P^{L+1}=P^1\,.$$
\end{definition}

\begin{theorem}\label{2502282100}
Let $\Omega$ be as in Definition~\ref{def_Omega} and let $\Omega_\varepsilon$ be defined as in \eqref{eq_Omega_eps}.
Let $\cI$ and $\scrC(\Omega_\varepsilon)$ be defined as in \eqref{eq_energyperforated} and \eqref{eq_competitorsperforatedfloat}, respectively.
\begin{enumerate}
\item[{$\mathrm{(I)}$}]
Let $\hat{v}\in\scrC(\Omega_\varepsilon)$. 
Then the following three statements are equivalent.
\begin{enumerate}
\item[{$\mathrm{(1)}$}] The function $\hat{v}$ is the unique solution to the problem
\begin{equation}\label{2502131930}
\min\big\{\cI_{\varepsilon}(v;\Omega_{\varepsilon}):v\in \scrC(\Omega_\varepsilon)\big\}.
\end{equation}
\item[{$\mathrm{(2)}$}] The function $\hat{v}$ satisfies the Euler--Lagrange equation for \eqref{2502131930} in weak form, namely
\begin{equation}\label{2408291139}
0=\frac{1+\nu}{E}
\int_{\Omega_{\varepsilon}}\!\! \big[\nabla^2 \hat{v}:\nabla^2\phi-\nu\Delta \hat{v}\,\Delta\phi\big]\,\ud x +\sum_{j=1}^J \ave_{\partial B_{\ce}^j} \!\langle\nabla \phi, \Pi(b^j)\rangle\,\ud\Huno +\sum_{k=1}^K \ave_{\partial B_\ce^k} \! s^k\phi\,\ud\Huno\,
\end{equation}
for every $\phi\in \scrC(\Omega_\varepsilon)$.
\item[{$\mathrm{(3)}$}] The function $\hat{v}$ is given by
\begin{equation}\label{eq:202411261130}
\hat{v}(x) = - \frac{E}{1-\nu^2} \langle \mathcal{M}^{-1} \kappa(x),\Phi\rangle, \qquad\text{for $x\in\Omega_\varepsilon$\,,}
\end{equation}
where $\Phi\coloneqq (s^1, s^1 \xi_1^{(1)} + b_2^1, s^1 \xi_2^{(1)} - b_1^1, 
\ldots, s^N, s^{N} \xi_1^{(N)} + b_2^N, s^{N} \xi_2^{(N)} - b_1^N)$ and
where $\kappa \coloneqq (\kappa^1_0, \kappa^1_1, \kappa^1_2, \ldots, \kappa^N_0, \kappa^N_1, \kappa^N_2)\colon\Omega_\varepsilon\to\mathbb{R}^{3N}$
is the vector whose elements are obtained from the following cell  
problems: for every $i=1,\ldots,N$, the functions $\kappa^i_0$ and $\kappa^i_r$ ($r=1,2$) are the unique solutions to, respectively,
\begin{equation}\label{eq:2410171611}
\!\!\!\! \begin{cases}
\Delta^2 \kappa_0 = 0 & \text{in $\Omega_{\varepsilon}$,} \\
\kappa_0 = \partial_n \kappa_0 = 0 & \text{on $\partial\Omega$,} \\
\kappa_0 = 1 & \text{on $\partial B_{\varepsilon}^i$\,,} \\
\partial_n \kappa_0 = 0 & \text{on $\partial B_{\varepsilon}^i$\,,} \\
\kappa_0 = \partial_n \kappa_0 = 0 & \text{on $\partial B_{\varepsilon}^I$, if $I \ne i$,} 
\end{cases}
\quad\text{and}\quad
\begin{cases}
\Delta^2 \kappa_r = 0 & \text{in $\Omega_{\varepsilon}$,} \\
\kappa_r = \partial_n \kappa_r = 0 & \text{on $\partial\Omega$,} \\
\kappa_r = x_r & \text{on $\partial B_{\varepsilon}^i$\,,} \\
\partial_n \kappa_r = n_r & \text{on $\partial B_{\varepsilon}^i$\,,} \\
\kappa_r = \partial_n \kappa_r = 0 & \text{on $\partial B_{\varepsilon}^I$, if $I \ne i$.}
\end{cases}	
\end{equation}
In~\eqref{eq:202411261130}, $\mathcal{M}$ is the $(3N)\times(3N)$ block matrix whose blocks $\mathcal{M}_{ij}$ ($i,j=1,\ldots,N$) are the $3\times3$ matrices defined as (see \cite[formula~(50)]{CFM2025})
\[
(\mathcal{M}_{ij})_{rs}\coloneqq \int_{\Omega_\varepsilon} \nabla^2\kappa^i_r(x):\nabla^2\kappa^j_s(x)\,\mathrm{d}x.
\]
\end{enumerate}

\item[{$\mathrm{(II)}$}] Letting $\hat{v}$ as in $\mathrm{(I)}$ and $\hat{\sigma} \coloneqq \sigma[\hat{v}] = \cof(\nabla^2 \hat{v})$, define $\hat{\epsilon}\in L^2(\Omega_\varepsilon;\mathbb{R}^{2\times 2}_{\sym})$ by
\begin{equation}\label{eq_cofattore}
\hat\epsilon \coloneqq \CC^{-1}\hat\sigma=\CC^{-1}\cof(\nabla^2 \hat{v}). 
\end{equation}
If $\hat{v}\in H^4(\Omega_\varepsilon)$, then the strain $\hat\epsilon\in H^2(\Omega_\varepsilon;\R^{2\times2}_{\sym})$  is the unique
solution to
\begin{equation}\label{2502131927}
\begin{cases}
\Div \CC\epsilon = 0 & \text{in $\Omega_{\varepsilon}$\,,}\\		
\CC\epsilon\, n = 0 & \text{on $\partial\Omega_{\varepsilon}$\,,} \\
\inc \epsilon = 0 & \text{in $\Omega_{\varepsilon}$\,,} \\
\displaystyle \int_{\partial B_\varepsilon^i} (\epsilon_{rq,c}-\epsilon_{qc,r})\,\ud x_q =s^i & \text{for $i = 1,\ldots, N$,}\\[3mm]
\displaystyle \int_{\partial B_\varepsilon^i} [\epsilon_{rc}-x_q(\epsilon_{rc,q}-\epsilon_{cq,r})]\,\ud x_c = b^i_r & \text{for $i = 1,\ldots, N$ and $r=1,2$,}
\end{cases}
\end{equation}
where the boundaries $\partial B_\varepsilon^i$ are oriented counter-clockwise.
\item[{$\mathrm{(III)}$}] Suppose that the unique solution $\hat\epsilon$ to \eqref{2502131927} belongs to $C^\infty(\Omega_\varepsilon;\R^{2\times2}_{\sym})$. Then, there exists 
$\hat{v}$   satisfying \eqref{eq_cofattore}
such that it also verifies $\mathrm{(1)}$, $\mathrm{(2)}$, and $\mathrm{(3)}$ in  $\mathrm{(I)}$.
\end{enumerate}
\end{theorem}

\begin{proof}
(I) We start by observing that the space $\scrC(\Omega_\varepsilon)$ is (weakly) closed in $H^2_0(\Omega)$, so that the minimum problem in \eqref{2502131930} admits a solution by applying the direct method of the calculus of variations; moreover, the solution is unique owing to the strict convexity of the functional~$\cI_{\varepsilon}$. 
Notice that the boundedness of minimizing sequences in $H^2_0(\Omega)$ is obtained thanks to Friedrichs's Inequality, which holds for Lipschitz domains.
Thus, implication (1)$\Rightarrow$(2) follows from imposing that the first variation $\delta \cI_{\varepsilon}(\hat{v};\Omega_\varepsilon)[\phi]$ vanish for every $\phi\in\scrC(\Omega_\varepsilon)$.
Viceversa, implication (2)$\Rightarrow$(1) follows from the strict convexity of the functional.

The proof that (1)$\Rightarrow$(3) can be found in \cite[Proposition~3.2]{CFM2025}. 
There, to conclude, a symmetry property of the trilinear form \eqref{sec:24241126922} 
was used (see \cite[Lemma~C.2]{CFM2025}), which relied on the fact that $\partial\Omega$ was of class $C^2$. In the present case, where $\partial\Omega$ is only Lipschitz continuous, the same symmetry property holds true and it is proved in Lemma~\ref{2909251140} (see formula \eqref{eq_C3}). 
On the contrary, the proof that (3)$\Rightarrow$(1) can be borrowed from \cite[Proposition~3.2]{CFM2025} as this part remains valid also for $\Omega$ as in Definition~\ref{def_Omega}.

\smallskip
 (II) To prove the second part of the theorem, one can integrate by parts the weak form of the Euler--Lagrange equations \eqref{2408291139}, thanks to the $H^4$--regularity of $\hat{v}$.
Since $\partial B_\varepsilon^i$ is a circle for every $i=1,\ldots,N$ and $\hat{\epsilon}\in H^2(\Omega_\varepsilon;\R^{2\times2}_{\sym})$ implies that the integrands are (at least) $L^1$ (see, \emph{e.g.}, \cite[Theorem~6.10]{DNPV2012}), the boundary integrals in \eqref{2502131927} are well defined and are obtained by appropriately choosing affine test functions $\phi\in\scrC(\Omega_\varepsilon)$ in \eqref{2408291139}.
More precisely, we refer to \cite[Proposition~3.3]{CFM2025}, which ensures that \eqref{2408291139} becomes \begin{equation}\label{eq_strong_EL_v_min_BC}
\begin{cases}
\displaystyle \frac{1-\nu^2}{E}\Delta^2 \hat{v}=0 & \text{in $L^2(\Omega_\varepsilon)$,}\\[3mm]
\displaystyle \frac{1-\nu^2}{E} \int_{\partial B_\varepsilon^i} \partial_n(\Delta \hat{v})\,\ud\Huno=s^i & \text{for every $i=1,\ldots,N$,}\\[3mm]
\displaystyle \frac{1-\nu^2}{E} \int_{\partial B_\varepsilon^i} \bigg(x_1\partial_t(\Delta \hat{v})-x_2\partial_n(\Delta \hat{v})+\frac{(\nabla^2 \hat{v}\,t)_1}{1-\nu}\bigg)\,\ud\Huno=b^i_1 & \text{for every $i=1,\ldots,N$,}\\[3mm]
\displaystyle \frac{1-\nu^2}{E} \int_{\partial B_\varepsilon^i} \bigg(x_1\partial_n(\Delta \hat{v})+x_2\partial_t(\Delta \hat{v})+\frac{(\nabla^2 \hat{v}\,t)_2}{1-\nu}\bigg)\,\ud\Huno=b^i_2 & \text{for every $i=1,\ldots,N$,}
\end{cases}
 \end{equation}
for every $i=1,\ldots,N$, where $n$ is the outer unit normal to $\Omega_\varepsilon$.  
The first equation in \eqref{eq_strong_EL_v_min_BC} corresponds to the third equation in \eqref{2502131927} thanks to the Airy transformation \eqref{eq_Airyoperatorgen}.
Then, thanks to \cite[Proposition~2.2]{CFM2025} the last three equations in \eqref{eq_strong_EL_v_min_BC} correspond to the last two equations in \eqref{2502131927}.
Since $\hat{v}\in H^4(\Omega_\varepsilon)\cap \scrC(\Omega_\varepsilon)$, again \eqref{eq_Airyoperatorgen}
ensures that the second equation in \eqref{2502131927} holds true; finally, the first equation in \eqref{2502131927} is trivially satisfied thanks to the fact that $\Div(\CC\hat{\epsilon})=\Div \hat{\sigma}=\Div(\cof(\nabla^2\hat{v}))\equiv0$.

\smallskip (III) 
The proof mirrors that of \cite[Theorem~3.3]{CFM2025}, using that the boundary condition $\sigma\,n = \CC \epsilon\,n = 0$ on $\partial\Omega_\varepsilon$ reduces to $\nabla^2 v\,t = 0$ on $\partial\Omega_\varepsilon$.
The geometric lemma contained in \cite[Proposition~A.2]{Cesana2024a} cannot be applied in our conditions, since $\partial\Omega$ is not globally of class $C^2$; for the present case, we resort to Proposition~\ref{prop202510111410}. 
\end{proof}

\begin{remark}\label{2512101322}
The following remarks and observations regarding Theorem~\ref{2502282100} are in order.
\begin{enumerate}
\item The result contained in this theorem is that
$$\text{[\,$\hat{v}$ solves \eqref{2502131930} $\Leftrightarrow$ $\hat\epsilon$ solves \eqref{2502131927}\,]\, $\Leftrightarrow$\, \eqref{eq_cofattore} holds,}$$
provided there is enough regularity (as per the assumption in (III)). 
This regularity is needed to ensure that the boundary integrals in \eqref{2502131927} are well defined.

\item Theorem~\ref{2502282100} above contains the results of \cite[Theorem~3.3]{CFM2025}, under weaker assumptions on the regularity of $\partial\Omega$ (therefore providing a stronger statement).
We point out that, in \cite[Theorem~3.3]{CFM2025}, the proof of part (III) is obtained with a different strategy: there, the smoothness of the boundary allows one to conclude by invoking classical elliptic regularity results. 
In contrast, in the present case, the proof can be achieved thanks to Proposition~\ref{prop202510111410} in the Appendix, which allows us to extend the result in \cite[Proposition~A.2]{Cesana2024a} to the case of a domain~$\Omega$ as in Definition~\ref{def_Omega}, where the boundary is globally Lipschitz and piecewise $C^2$.

\item    
From an algorithmic  perspective, Theorem \ref{2502282100} (Part (I)) suggests the following strategy: first solve the $3N$ cell problems \eqref{eq:2410171611} once for a given geometry and core radius, then reconstruct the Airy potential $\hat{v}$ for arbitrary sets of Burgers vectors and Frank angles via $\eqref{eq:202411261130}$. This decouples geometry from defect data and turns the equilibrium problem into a finite-dimensional algebraic computation in postprocessing (see Section \ref{sec:202509161701}).
Another important advantage and application of these formulas is that they simplify the treatment of the boundary conditions for the stress. By enforcing Dirichlet boundary conditions, they transform the minimization problem $\eqref{2502131930}$ in $\scrC(\Omega_\varepsilon)$, which is well posed but involves unknown boundary values of the Airy potential on the cores, into a finite dimensional algebraic minimization. See also $\cite{CARBONARA7, PETROLO20042471}$.

\item By a standard localization argument and resorting to the Sobolev Embedding Theorem, one can see that, for every $i\in\{1,\ldots,N\}$, the solutions $\kappa^i_0$ and $\kappa^i_r$ ($r=1,2$) to the problems in \eqref{eq:2410171611} are (at least) of class $C^{2,\beta}$ (for any $\beta\in(0,1)$) up to $\partial B_\varepsilon^i$\,.
As a consequence, $\nabla^2\hat v\,t$ is well defined pointwise on $\partial B_\varepsilon^i$ for every $i\in\{1,\ldots,N\}$, and so is the normal stress $\hat\sigma\,n$.
\end{enumerate}
\end{remark}

 \section{Numerical formulation}\label{sec:202509161701}
 
As highlighted in Remark \ref{2512101322}-3 above, the $3N$ cell problems \eqref{eq:2410171611} depend only on the geometry of the system, in particular on the locations of the defects and on the value of $\varepsilon$. They do not depend on the type of defect (dislocation vs.~disclination), nor on the magnitudes of the Burgers vectors or Frank angles. 
After the $3N$ cell problems have been solved for a given distribution of cores in the elastic body, they can be reused to compute, in real time, the Airy stress potential generated by any configuration of $N$ defects placed at those locations and with core radius $\varepsilon$, for any choice of Burgers vectors and Frank angles
via formula~\eqref{eq:202411261130}.

The prototype of the cell problem corresponds to a fourth order biharmonic equation in an $H^2$ scalar unknown. We solve this problem using the $C^0$ interior penalty discontinuous Galerkin ($C^0$--IPDG) method.
$C^0$--IPDG methods are nonconforming finite element (FE) schemes. Unlike $H^2$ conforming FE methods, which require globally $C^1$ smooth shape functions, use only $C^0$ continuous elements to approximate fourth order elliptic problems. This is achieved by modifying the weak formulation and introducing a penalization term that accounts for jumps in the gradients of the shape functions across mesh facets.

The main advantage of $C^0$--IPDG methods over $H^2$ conforming FE methods lies in their ease of implementation. While $H^2$ conforming FE methods require specialized higher-order elements, such as the Argyris or Hsieh--Clough--Tocher elements, the $C^0$--IPDG approach can be implemented using standard Lagrangian elements. For a review of such higher order elements we refer the reader to \cite{PAPANICOLOPULOS2012, PAPANICOLOPULOS2013}.
Mixing finite element methods offers a valid alternative to tackle higher order problems. For a review of $H^2$ conforming mixing finite element methods that use $C^0$ globally continuous piecewise polynomial base functions we refer to  \cite{AINSWORTH2024,BALASUNDARAM1984,FARRELL2022}.
For a comprehensive review of DG methods for second order elliptic problems, we refer the reader to~\cite{Arnold82,Arnold2002}, and for the $C^0$--IPDG formulation of the biharmonic problem, to~\cite{Georgoulis08,Brenner2012}.

Let  $\mathcal{T}$  denote a regular triangulation of  $\Omegaeps \subset \mathbb{R}^2$ into closed triangles. Define \( \edgbdouter \) as the set of boundary edges covering $\partial \Omega$, \( \edgin \) as the set of internal edges, and $\edgbdinner{i}$ as the set of boundary edges covering $\partial B_{\varepsilon}^i$ ($i = 1$,$\hdots, N$), so that \( \mathcal{E} \coloneqq \edgin \cup \edgbdouter \cup \big(\bigcup_{i=1}^N \edgbdinner{i}
\big)\) is the collection of all edges. For  $T \in \mathcal{T}$, let  $\mathrm{P}_3(T)$  be the space of polynomials of degree at most three defined on  $T$. For $I \in \{ 1, \hdots, N\}$ and $p \in \{0, 1, 2\}$ we define the following $3N$ finite element spaces
\begin{multline}
\mathcal{V}_p^I(\mathcal{T}) \coloneqq \Big \{ u \in C^{0}(\overline{\Omega}_\varepsilon) : u|_{T} \in \text{P}_3(T) \quad \forall\, T \in \mathcal{T}, \, u = 0 \text{ on } \edgbdouter, \\
u = 0 \text{ on } \edgbdinner{i} \text{ for every } i \in \{1, \ldots, N\}\setminus\{I\}, u = u_p \text{ on } \edgbdinner{I} \Big \}.
\end{multline}
These spaces are used to solve the cell formula problems associated with $\kappa_p^I$. 
For a given index $I$, only the corresponding hole is active: 
the boundary conditions on its edges are 
\[
u_0 = 1, \quad u_1 = x_1, \quad u_2 = x_2,
\] 
while all remaining holes satisfy homogeneous boundary conditions. 

For every edge $e \in \mathcal{E}$   and every scalar field $u \colon \overline{\Omega}_\varepsilon \to \mathbb{R}$, the jump and average operators are defined as follows
\begin{equation}\label{2412291421}
\llbracket u \rrbracket_{e} \coloneqq
\begin{cases}
u_{+}|_{e} - u_{-}|_{e} & \text{ if } e \in \edgin\,,  \\
u|_{e}                  & \text{ if } e \in  \edgbdouter \cup_{i=1}^N \edgbdinner{i} 
\end{cases}
\qquad
\{ u \}_{e} \coloneqq
\begin{cases}
\displaystyle{\frac{u_{+}|_{e} + u_{-}|_{e}}{2} } & \text{ if } e \in \edgin\,,  \\
u|_{e} & \text{ if } e \in \edgbdouter \cup_{i=1}^N \edgbdinner{i} 
\end{cases}
\end{equation}
where $u_{+} = u |_{T^{+}}$, $u_{-} = u |_{T^{-}}$, and $T^{+}$,
$T^{-}$ are two closed triangles whose intersection is   $e \in \edgin$\,. Definitions \eqref{2412291421} should be applied componentwise when $u$ is a vector or a matrix. 
For any $T \in \mathcal{\mathcal{T}}$, let $\eta \coloneqq \text{diam}(T)$.
We   take $n_e$ to be the outward unit normal on $e \in \edgin$, pointing from~$T^{-}$ to~$T^{+}$. 
Notice that $\llbracket \partial_{n_e} u \rrbracket$ and $\{ \partial_{n_e n_e} u \}$ are independent of the particular labeling of two adjacent triangles $T^{+}$ and $T^{-}$. The vector $n_e$ denotes the outward unit normal when $e \in \edgbdouter$ or $e \in \edgbdinner{i}$\,, for $i \in \{1, \hdots, N \}$.

 We are now in the position to show the $C^0$--IPDG formulation for the $3N$ cell formula problems \eqref{eq:2410171611}. For a given $p \in \{0, 1, 2\}$ and a given $I \in \{1, \hdots, N\}$, find $\kappa_p^I$ $\in \mathcal{V}_p^I(\mathcal{T})$ such that
 \begin{equation}
    \label{eq:202412271313}
    A(\kappa_p^I, \varphi) =  \ell(\varphi)
    \quad \forall \varphi \in \mathcal{C}(\mathcal{T}),
\end{equation}
where
\begin{align}
    \label{eq:20260191724}
    &A(\kappa_p^I, \varphi) \coloneqq \sum\limits_{T \in \mathcal{T}} \left( \int_{T} \nabla^2 \kappa_p^I : \nabla^2 \varphi \, \ud \xi \right)  - \sum_{e \in \mathcal{E}} \Big( \int_{e} \llbracket \nabla \kappa_p^I \cdot n_e \rrbracket_{e}  \{ \nabla^2 \varphi  n_e \cdot n_e \}_{e} \, \ud \Huno + \\
    &\int_{e} \llbracket \nabla \varphi \cdot n_e \rrbracket_{e}  \{ \nabla^2  \kappa_p^I n_e \cdot n_e \}_{e} \, \ud \Huno - \frac{\IP}{\{ \eta \}_e} \int_{e} \llbracket \nabla \kappa_p^I \cdot n_e \rrbracket_{e} \llbracket \nabla \varphi \cdot n_e \rrbracket_{e} \, \ud \Huno \Big), \nonumber
\end{align}
\begin{equation}
    \label{eq:20260116}
    \ell(\varphi) \coloneqq 
    \hspace{-0.5em} \sum_{e \in \edgbdinner{I} } \left( \frac{\IP}{\{ \eta \}_e} \int_{e} g_p  \partial_{n_{e}} \varphi \, \ud \Huno - \int_{e} g_p  \partial_{n_{e} n_{e}} \varphi \, \ud \Huno \right),
\end{equation}
with $g_0 = 0$, $g_1 = (n_e)_1$\,, $g_2 = (n_e)_2$.
Here, 
\begin{equation}
\mathcal{C}(\mathcal{T}) \coloneqq \left \{ u \in C^{0}(\overline{\Omegaeps}) : u|_{T} \in \text{P}_3(T) \quad \forall T \in \mathcal{T}, \, u = 0 \text{ on } \edgbdouter \cup \left(\bigcup_{i=1}^N \edgbdinner{i} \right) \right \},
\nonumber
\end{equation}
denotes the space of test functions that vanish on all boundaries.
 
We employ the
DolfinX/FEniCSx framework \cite{BarattaEtal2023} for finite element discretization, and PETSc \cite{petsc-tao-users-manual} for efficient and scalable linear algebra operations.  Our implementation is available at \url{https://github.com/EdoardoF1993/planar_elasticity_dislocations_disclinations}.
 
\subsection{Validation of the numerical implementation}

We validate the numerical framework presented in~\eqref{eq:202412271313}
by comparing numerical solutions with exact formulas available in radial geometry.
Throughout this section, 
 $\Omega$ is the unit disk centered at the origin so that $\Omegaeps = B_1(0) \setminus \overline{B}_{\varepsilon}(0)$. We set $\nu = 0.15$ and $E=1$ and use third-order polynomial shape functions. The interior penalty parameter has been selected via trial-and-error and set equal to $\IP = 10$ for all simulations discussed in this paper. 
Comparable values are used in standard implementations, see for instance \cite{Brenner:Von_Karman} and the DOLFINx biharmonic demo available at \href{https://docs.fenicsproject.org/dolfinx/v0.7.2/python/demos/demo_biharmonic.html}{docs.fenicsproject.org}.

\paragraph{Validation on the solutions to cell formula problems \eqref{eq:2410171611}.}
\label{sec:202509161702}

 In this paragraph, we construct numerical solutions to \eqref{eq:2410171611}, and  compare 
 them with the corresponding analytical expressions.
In  $\Omega_{\varepsilon}$\,, the three (unique) solutions to the corresponding three cell formula problems read  
\begin{equation}
    \label{eq:202503261033}
    \kappa^1_{0}(x) = C_1 (|x|^2 - 1) + C_2 x^2 \ln |x| + C_3 \ln|x|, \quad \kappa^1_{1}(x) = P(x) \frac{x_1}{|x|} \quad \text{and} \quad \kappa^1_{2}(x) = P(x) \frac{x_2}{|x|},
\end{equation}
where $P(x) \coloneqq \frac{1}{8} C_4 |x|^3 + C_5 |x| + C_6 |x|^{-1} +\frac{1}{2} C_7 |x| \ln |x|$, and 
\begin{align}
    \label{eq:20250510150}
    &C_1 \coloneqq -\frac{C_2 + C_3}{2}, \quad C_2 \coloneqq 2\frac{1 - \varepsilon ^2}{\left(1 - \varepsilon^2 \right)^2-4 \varepsilon ^2 \ln ^2(\varepsilon )}, \quad C_3 \coloneqq -\frac{4 \varepsilon ^2 \ln (\varepsilon )}{\left(1 - \varepsilon ^2 \right)^2-4 \varepsilon ^2 \ln ^2(\varepsilon )},\\
    &C_4 \coloneqq -\frac{4}{1 -\varepsilon^2 + \left(\varepsilon^2 + 1 \right) \ln(\varepsilon )}, \quad C_5 \coloneqq -C_4 \frac{ 1 + \varepsilon^2}{2}, \quad C_6 \coloneqq - \frac{1-\varepsilon^2}{8} C_4, \quad C_7 \coloneqq - \frac{\varepsilon^2}{8} C_4. \notag
\end{align}
Figure~\ref{fig:202507281047} shows a superposition of the exact solutions and their numerical approximations, with agreement observed in both energy norms and pointwise values for various values of the core radius
  \(\varepsilon \in (0, 1)\).

An important point in our mathematical theory is the imposition of the affine boundary conditions for all boundary value problems for the Airy function. As explained in Section \ref{2511271518}, see also Theorem \ref{2502282100} for the mathematical framework and Remark \ref{2512101322}-4 to Theorem \ref{2502282100}, this guarantees that the normal stress vanishes. We systematically verify that this zero stress condition is obtained in our numerical simulations, see Fig. \ref{fig:202507281047} bottom right.

Owing to the radial symmetry of the problem, $\kappa_0^1$ and $\kappa_0^2$ differ only by a 
rotation of the domain. Consequently, their $H^2$--norms are identical. For this reason, the comparison between the numerical and analytical solutions for the cell problem associated with $\kappa_0^2$ is omitted.

\begin{figure}[h!]
    \centering
    \begin{minipage}[b]{0.49\textwidth}
        \centering
        \includegraphics[width=\textwidth]{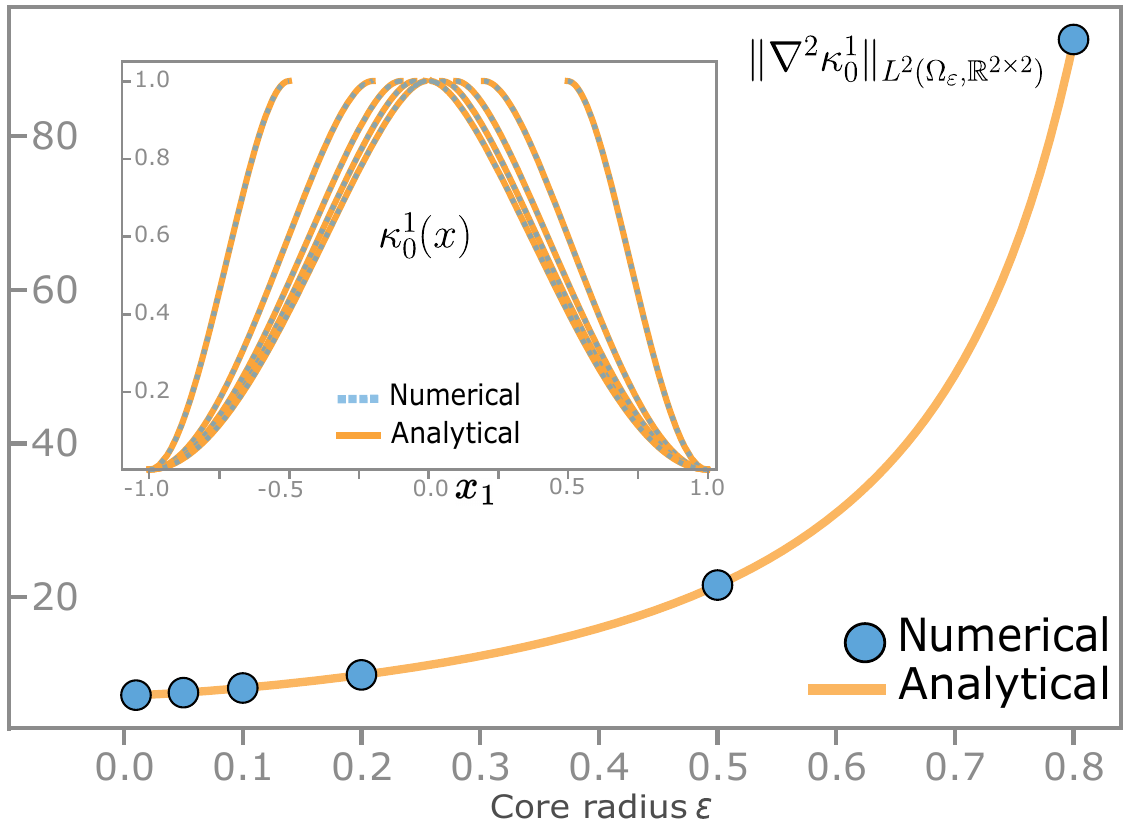}
    \end{minipage}
    \begin{minipage}[b]{0.49\textwidth}
        \centering
        \includegraphics[width=\textwidth]{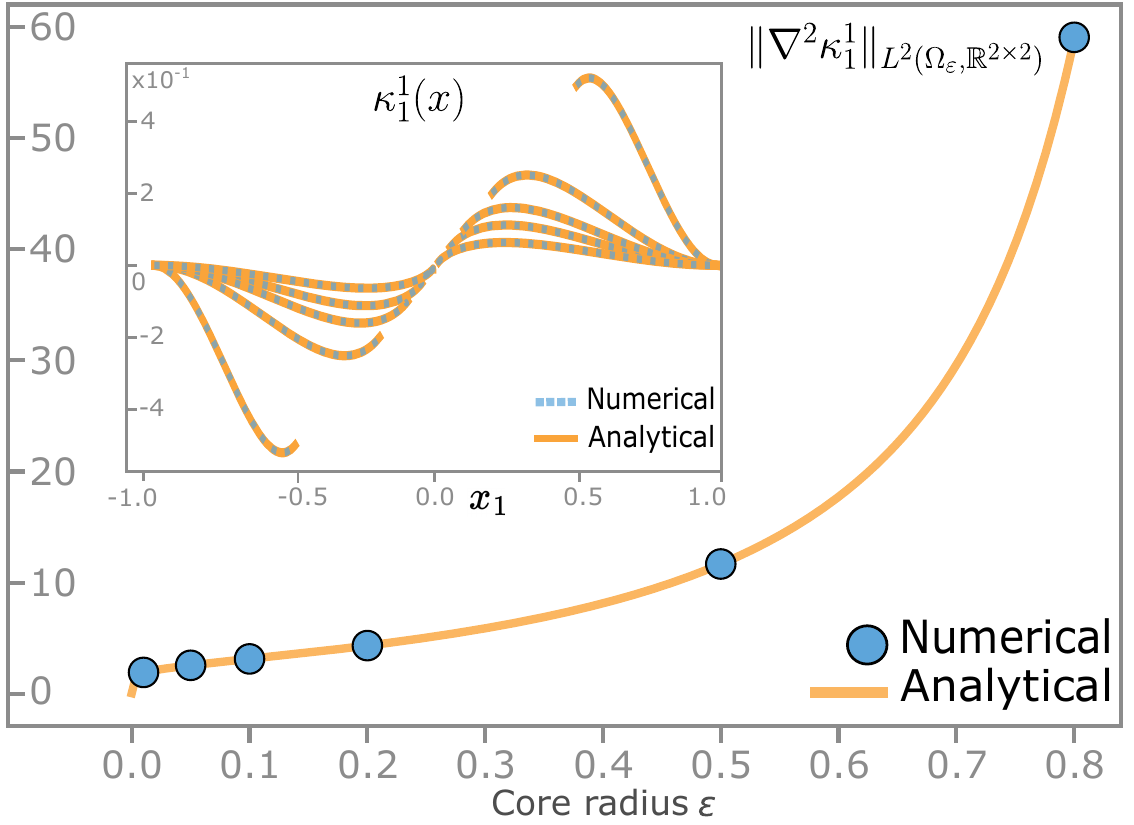}
    \end{minipage}
 
    \begin{minipage}[b]{0.49\textwidth}
        \centering
        \includegraphics[width=\textwidth]{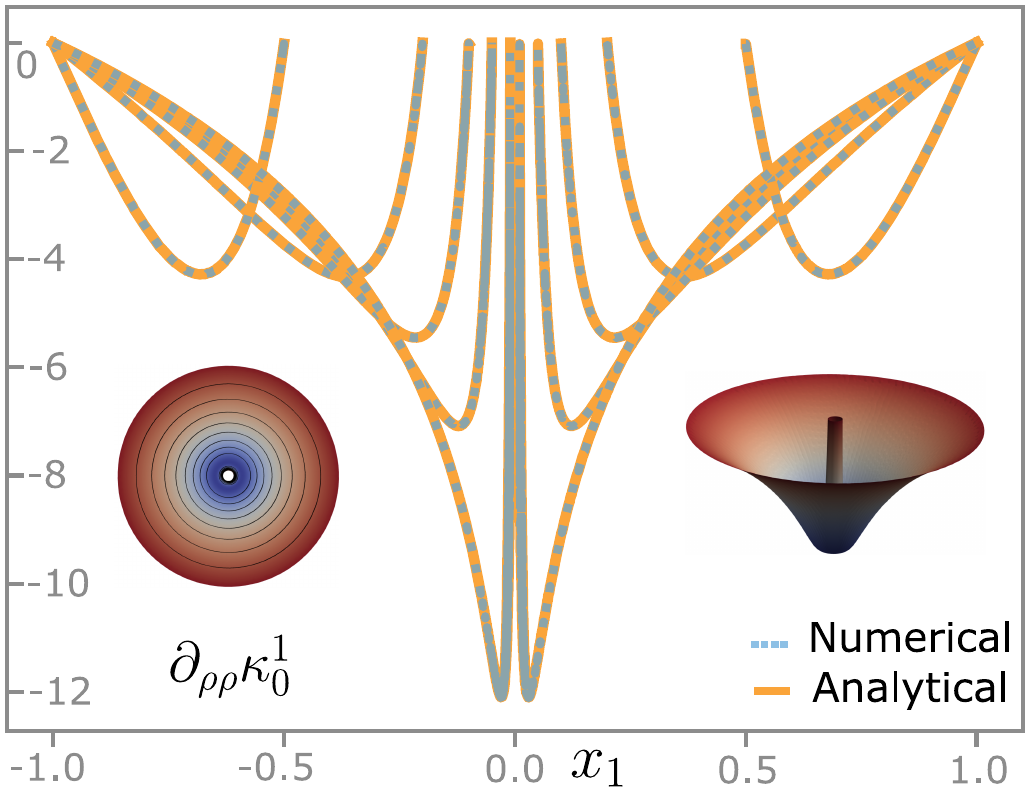}
    \end{minipage}
    \begin{minipage}[b]{0.49\textwidth}
        \centering
        \includegraphics[width=\textwidth]{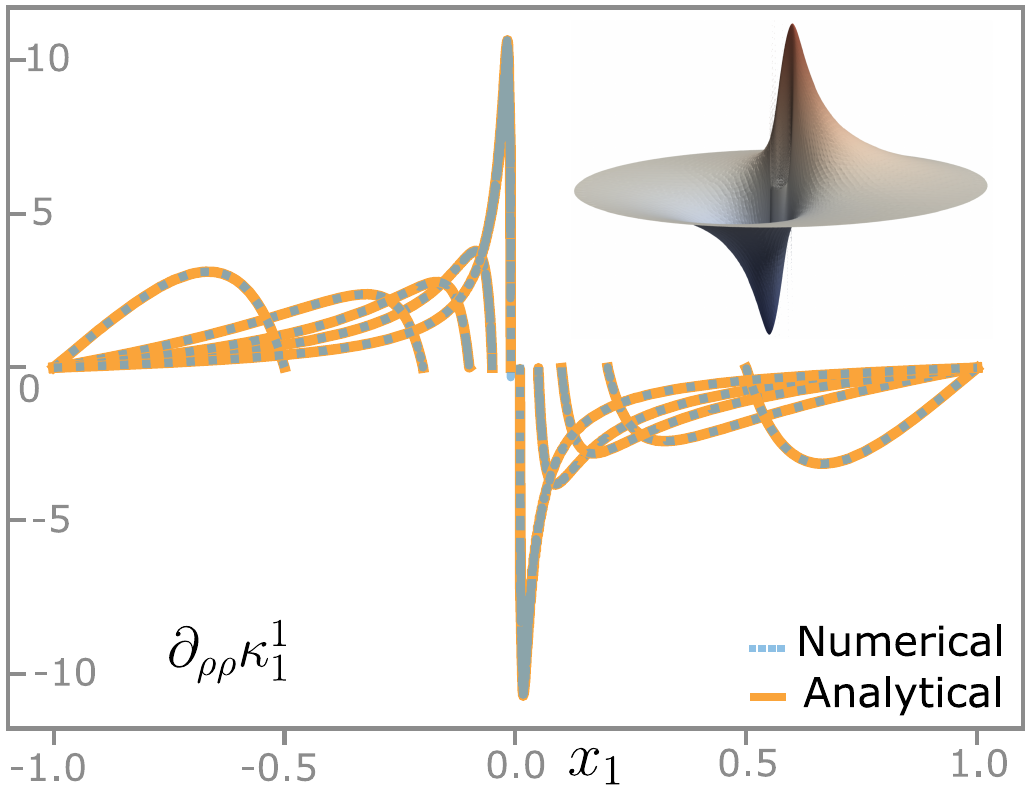}
    \end{minipage}
    
    \caption{Cell formulas, comparison of analytical and numerical solutions. 
    The yellow curves represent the analytical expressions computed using Eqs. \eqref{eq:202503261033}, while the blue circles are computed numerically. 
    Top row: $L^2$--norm of  $\nabla^2 \kappa_0^1$ (left) and $\nabla^2 \kappa_1^1$ (right) for different values of $\varepsilon \in (0, 1)$. 
    Inset: profiles of $\kappa_0^1$ (left) and $\kappa_1^1$ (right) for several values of $\varepsilon$, with $x_1 \in (-1, 1) \setminus [-\varepsilon, \varepsilon]$ and $x_2 = 0$. 
     Bottom row: profiles of $\partial_{\rho\rho} \kappa_0^1$ (left) and $\partial_{\rho\rho} \kappa_1^1$ (right) for different values of $\varepsilon$, with $x_1 \in (-1, 1) \setminus [-\varepsilon, \varepsilon]$, $x_2 = 0$, and $\rho \coloneqq |x|$. 
     Inset: three-dimensional and contour plots of $\partial_{\rho\rho} \kappa_0^1$ (left) and $\partial_{\rho\rho} \kappa_1^1$ (right) computed for $\varepsilon = 0.05$. A mesh size of $\eta = 0.08$ is used, with further refinement near the cores.}
    \label{fig:202507281047}
\end{figure}

\paragraph{Validation on configurations with disclinations and dislocations.}
\label{sec:202509161703}

In this paragraph,  
we obtain numerical solutions for configurations that contain a single disclination placed at different positions of the core. These solutions are compared with the exact solutions for the same configuration of disclinations on the full disk $B_1(0)$. The aim of this comparison is to show that the numerical model can approximate the stress fields and the energies of the reference model when $\varepsilon$ is sufficiently small.

The analytical expression of the Airy   
function induced by a single wedge disclination of Frank angle $s$, located at a generic $y^{(k)} \in B_1(0)$ is 
\begin{equation}
\label{eq:202412271400}
\begin{split}
\hat v(x) =&\, - \frac{s E}{(1-\nu^2)} \mathscr{G}(x; y^{(k)}), \\
\mathscr{G}(x; y) \coloneqq&\, \begin{cases}
\displaystyle \frac{1}{16\pi} \left(\left(1-|x|^2\right) \left(1-|y|^2\right)+|x-y|^2 \displaystyle \ln  \frac{|x-y|^2}{|x|^2 |y|^2 - 2x \cdot y + 1} \right) \quad & \text{for } x \ne y, \\
0 \quad & \text{otherwise,}
\end{cases}
\end{split}
\end{equation}
where $\mathscr{G}$ is the Green function for the bilaplacian operator in the unit disk, see \cite{PolyharmonicGreenFunction}.
Results are shown in Figure~\ref{fig:202509121636} for two values of $\varepsilon$, namely $\varepsilon = 0.01$ on the left and $\varepsilon = 0.05$ on the right. 
The comparison between these two cases indicates that, for values of $\varepsilon$ of the order of $10^{-2}$ or smaller, the regularized model accurately captures the   behavior of the continuum model.

Observe
 that the norm of the stress decreases as the disclination approaches the boundary of~$\Omega$, accompanied by a corresponding reduction in the magnitude of the Airy stress function. The same behavior has been observed in continuum models of dissipative evolution of disclinations defined on the full domain~$\Omega$, where both the elastic energy and the stress field tend to zero as the disclination moves toward the boundary of a circular domain, until it is eventually expelled, see \cite{CGMP2025}.

\begin{figure}[h!]
    \centering

    \begin{subfigure}[b]{0.49\textwidth}
        \centering
        \begin{tikzpicture}[remember picture,overlay]
            \node[anchor=north west] at (2.48,-0.5) {\small Eq.\hspace{-0.1cm}~\eqref{eq:202412271400}};
            \node[anchor=north west] at (-3,-4.85) {\tiny Eq.\hspace{-0.1cm}~\eqref{eq:202412271400}};
        \end{tikzpicture}
        \includegraphics[width=\textwidth]{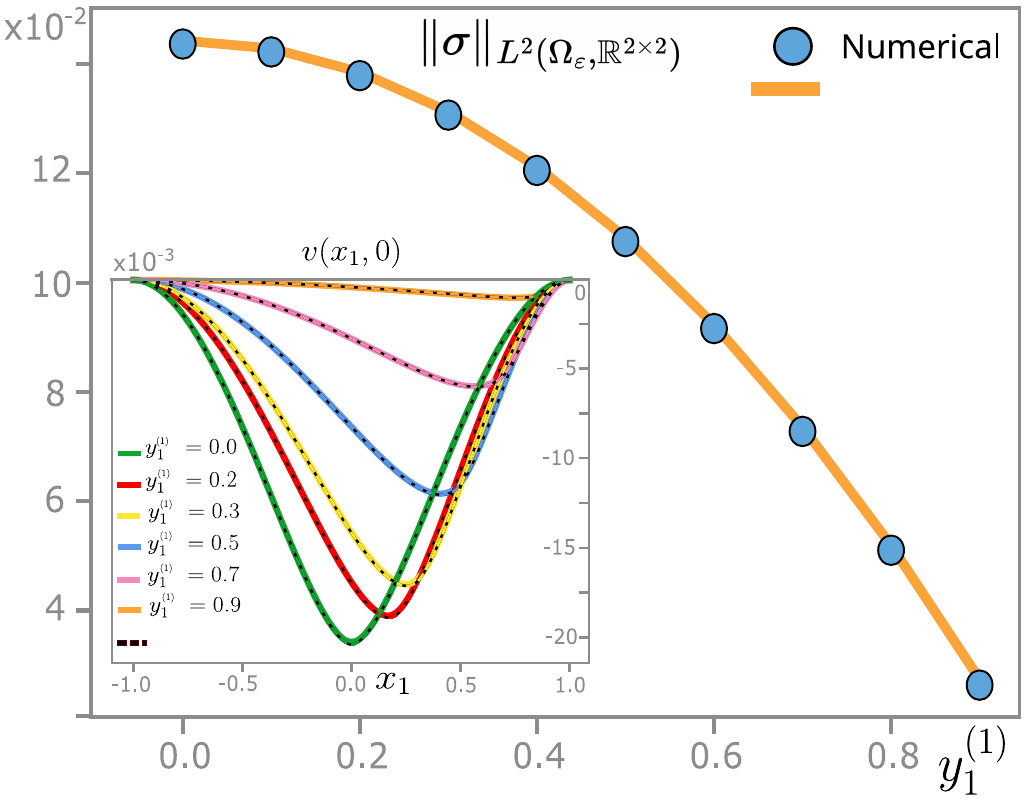}
    \end{subfigure}
    \hfill
    \begin{subfigure}[b]{0.49\textwidth}
        \centering
        \begin{tikzpicture}[remember picture,overlay]
            \node[anchor=north west] at (2.56,-0.5) {\small Eq.\hspace{-0.1cm}~\eqref{eq:202412271400}};
             \node[anchor=north west] at (-2.9,-4.9) {\tiny Eq.\hspace{-0.1cm}~\eqref{eq:202412271400}};
        \end{tikzpicture}
        \includegraphics[width=\textwidth]{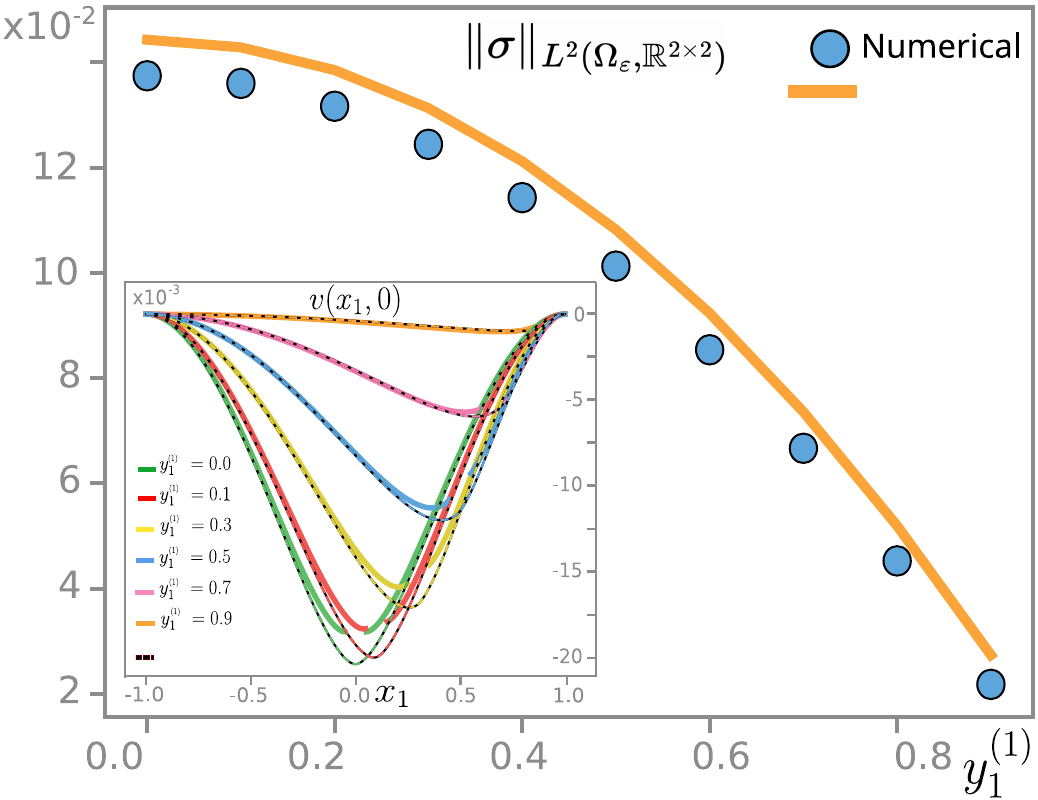}
    \end{subfigure}
    \caption{
    Isolated disclination benchmarks. Configurations with an isolated disclination at different locations.
    Left: $L^2$--norm of $\sigma$ for a single positive disclination located on different positions on the $x_1$ axis. The yellow curve indicates the analytical solution, the blue dots indicate the numerical solution ($\varepsilon = 0.01$). The inset shows the profile of the Airy potential for both the analytical solution (black dashed curve) and numerical solution (colored curve) obtained using $\varepsilon = 0.01$. Right: profiles of the Airy potential for $\varepsilon = 0.05$. A mesh size of $\eta = 0.02$ is used, with additional refinement in the vicinity of the cores.}
    \label{fig:202509121636}
\end{figure}
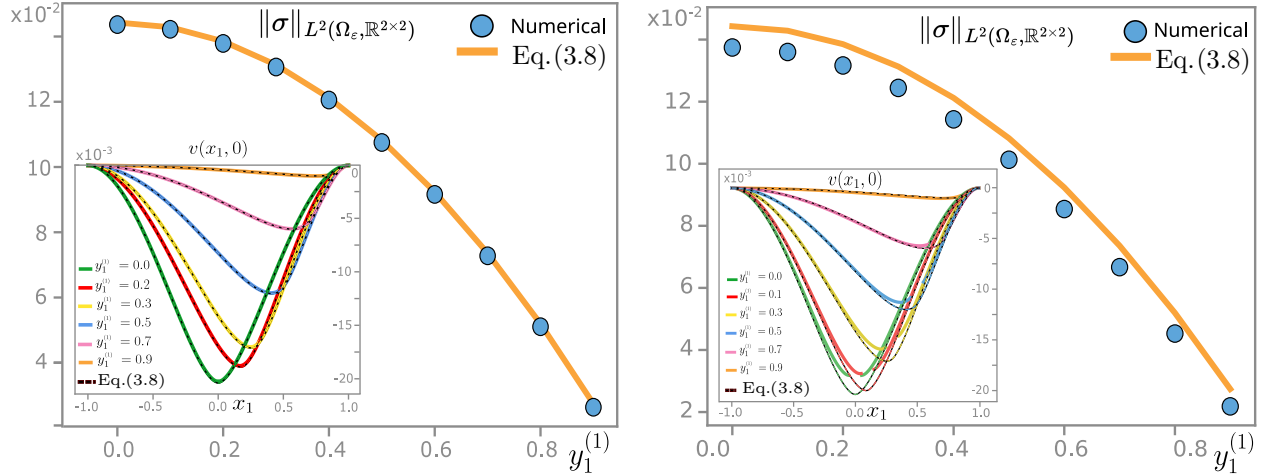

As a second numerical test we consider two wedge disclinations of 
 opposite sign located on the~$x_{1}$ axis at equal distance from the origin. The analytical expression of the corresponding Airy stress function  is obtained from Eq. \eqref{eq:202412271400}, using the linearity of the problem. The numerical model reproduces the stress field  accurately  for $\varepsilon = 0.01$, see the left panel of Figure \ref{fig:202509121637}.
The profile of the $L^{2}$ norm of the stress shows that the stress, and consequently also the mechanical energy, tend to vanish as the disclinations are located closer to $\partial \Omega$, as well as when they are located  closer to the origin of the domain, where they eventually coincide and annihilate.
We remark that the numerical approximations of the stresses considered here correspond to the quantities denoted by $\hat{\sigma}$ in Theorem \ref{2502282100}. For ease of notation, in this section and in what follows, we drop the hat symbol.
Stress values are reconstructed from the Airy potential via Eq. \eqref{eq_Airyoperatorgen}.

As our third and final test we consider a single edge dislocation centered at $x = 0$. We report the analytical expression of the corresponding Airy stress function
\begin{equation}
\label{eq:202511201328}
    \hat v(x) = b \times \dfrac{E}{1-\nu^2} \mathcal{F}(x), \qquad \mathcal{F}(x) \coloneqq
    \begin{cases}
    \dfrac{|x|^{2} - \,\ln |x|^{2} - 1}{8\pi}\, x, & x \neq 0, \\[6pt]
    0, & x = 0 .
    \end{cases}
\end{equation}
A direct computation shows that  $\Delta^2 \hat v(x) = - b \times \Delta^2 v_D(x)$, where the function $v_D$ (see \cite[Formula (2.22b)]{CFM2025}) satisfies the problem $\Delta^2 \left(e \times v_D \right) = e \times \nabla \de_0$ for any $e\in\mathbb{R}^2$, subject to the boundary conditions $\hat v(x) = \partial_n \hat v(x) = 0$ on $\partial B_1(0)$, so that $\Delta^2 \hat v = - \frac{E }{1-\nu^2}\, b \times \nabla \de_0$. 
The right panel of Figure~\ref{fig:202509121637} shows that the Airy stress function obtained numerically for several values of $\varepsilon$ approaches  the analytical expression in the limit   $\varepsilon \to 0$. The inset also shows that the square of the stress (and, consequently, the mechanical energy of the system) diverges proportionally to $|\log\varepsilon|$, as expected for an edge dislocation.

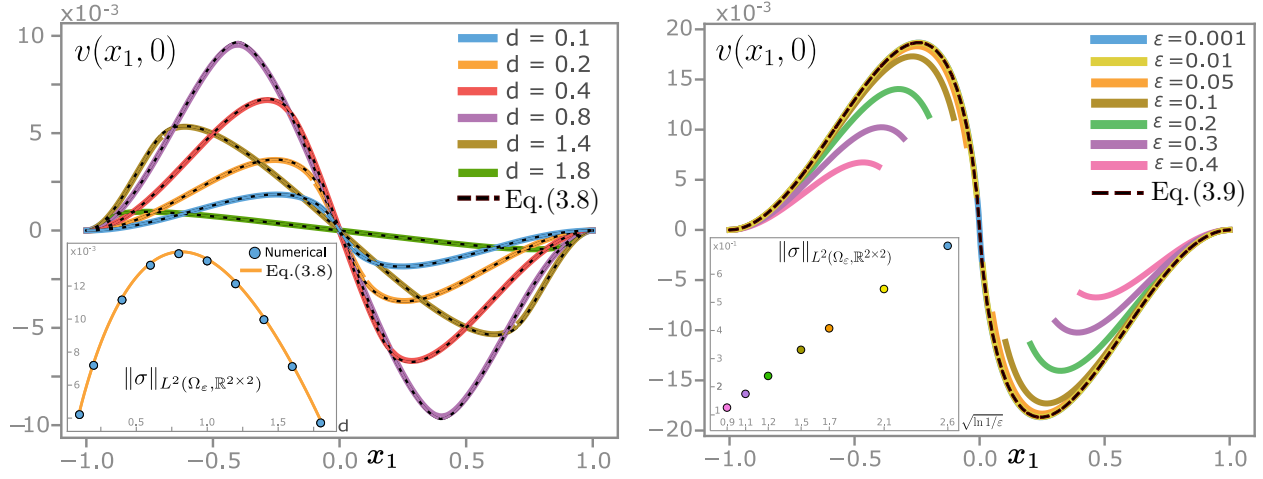
\begin{figure}[h!]
    \centering

    \begin{subfigure}[b]{0.49\textwidth}
        \centering
        \begin{tikzpicture}[remember picture,overlay]
            \node[anchor=north west] at (2.38,-2.35) {\small Eq.\hspace{-0.1cm}~\eqref{eq:202412271400}};
             \node[anchor=north west] at (-0.78,-3.35) {\tiny Eq.\hspace{-0.1cm}~\eqref{eq:202412271400}};
        \end{tikzpicture}
        \includegraphics[width=\textwidth]{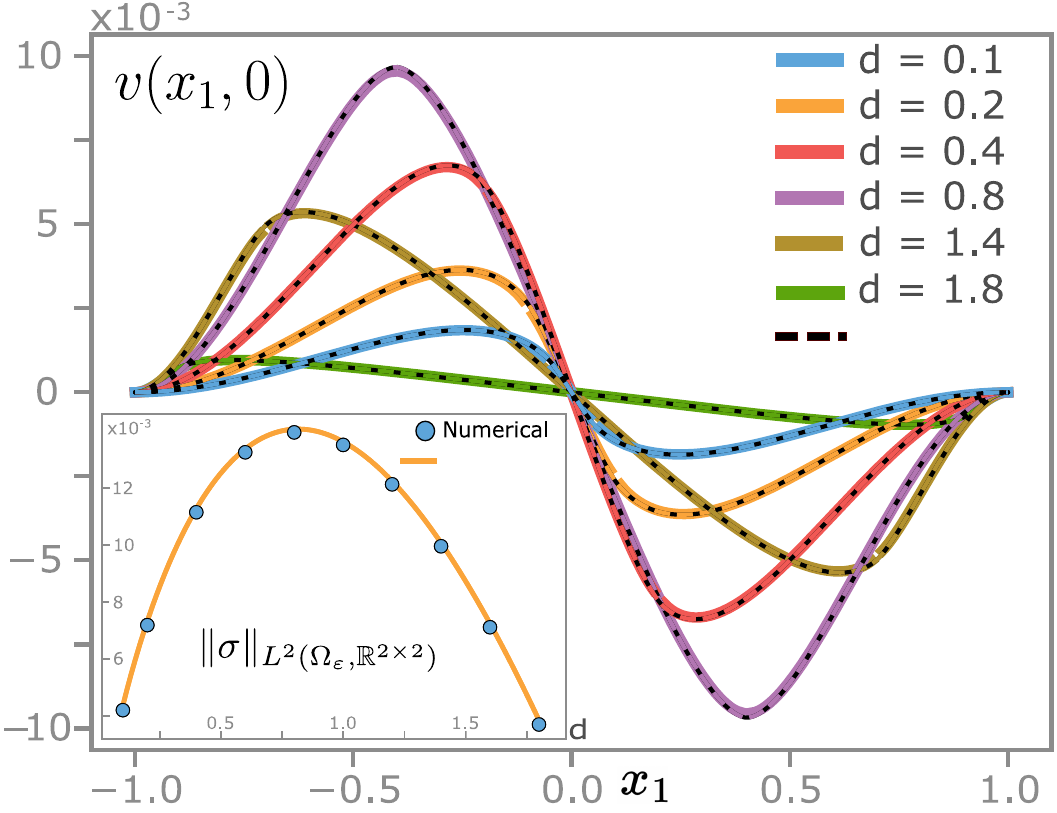}
    \end{subfigure}
    \hfill
    \begin{subfigure}[b]{0.49\textwidth}
        \centering
        \begin{tikzpicture}[remember picture,overlay]
            \node[anchor=north west] at (2.5,-2.3) {\small Eq.\hspace{-0.1cm}~\eqref{eq:202511201328}};
        \end{tikzpicture}
        \includegraphics[width=\textwidth]{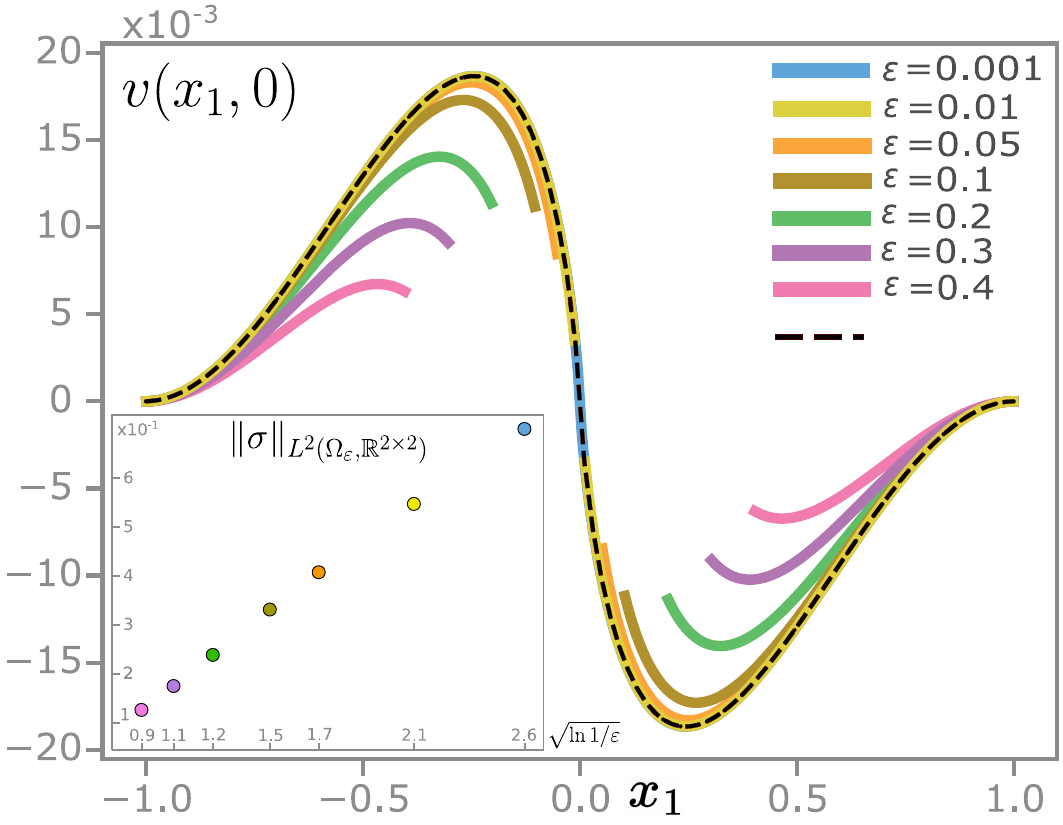}
    \end{subfigure}

    \caption{Left: profiles of the Airy potential for a pair of wedge disclinations with opposite Frank angles, symmetrically located along the $x_1$-axis, both at a distance $\tfrac{d}{2}$ from the origin. 
    The black dashed curve represents the solution obtained using Eq.~\eqref{eq:202412271400}, while the continuous colored curve shows the numerical solution for $\varepsilon = 0.01$. 
    The inset displays the $L^2$-norm of $\sigma$ for different values of $d$. 
    Right: profiles of the Airy potential computed for an edge dislocation, for various values of $\varepsilon$, compared with the analytical solution (Eq.~\eqref{eq:202511201328}). 
    The inset shows values of stress norms obtained numerically for different values of $\varepsilon$. A mesh size of $\eta = 0.02$ is used, with additional refinement in the vicinity of the cores.}   
    \label{fig:202509121637}
\end{figure}

\section{Illustration of complex interactions}\label{2512112300}

We present simulations of relevant configurations of isolated wedge    disclinations, disclination dipoles, and edge dislocations. 
 In what follows,
we consider a core size \( \varepsilon = 0.01 \), unless otherwise stated.
We consider Burgers vectors of magnitude \( |b| = 10^{-3} \) and disclinations with Frank angle \( s = \pm 0.1 \). The numerical values are nondimensional and chosen for illustrative purposes only. Their absolute magnitudes reflect relative scalings between parameters, for example \( |b| \approx \varepsilon \), and are not intended to represent any specific material. 

\subsection{Stack of edge dislocations and isolated wedge disclination}
\label{sec:202509301227}

We  consider the interaction between a fixed stack (pile--up) of edge dislocations and an isolated wedge disclination positioned at varying distances. 
The purpose of this study is to examine how the mechanical stress   generated by the dislocation–disclination system depends on the sign of the Frank angle and on the orientations of the dislocations.
The computational domain is the square
\(\Omega = (-1,1)\times(-1,1)\).
The three edge dislocations are located at \(x_1 = -0.05\), \(x_1 = -0.025\), and \(x_1 = 0\), respectively, and the wedge disclination is positioned on the positive \(x_1\) axis at a distance \(d\) from the origin, with \(d\) treated as a control parameter. The corresponding perforated domain \(\Omega_{\varepsilon}\) is obtained by removing three closed disjoint balls of radius \(\varepsilon\) centered at the dislocation cores and a fourth closed ball of the same radius centered at the disclination.

\begin{figure}[h!]
\centering
\begin{minipage}{0.47\textwidth}
\centering
\begin{tikzpicture}
    \node[anchor=south west,inner sep=0] (img) 
        at (0,0) {\includegraphics[width=\linewidth]{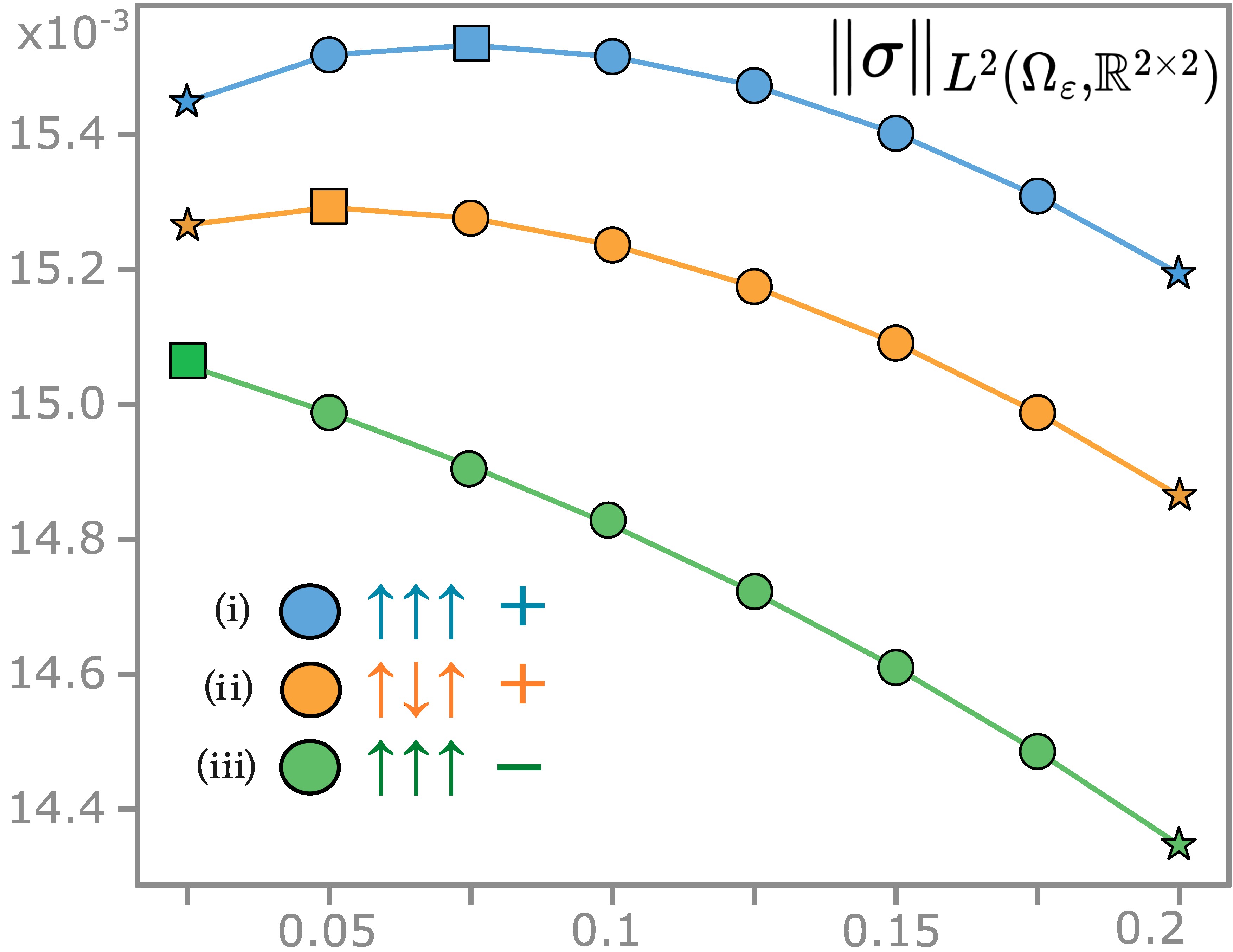}};
    \node[anchor=south] at ($(img.south)+(4.1,-0.1)$) {$y_1^{(1)}$};
\end{tikzpicture}
\end{minipage}
\hfill
\begin{minipage}{0.47\textwidth}
\centering
\begin{tikzpicture}
    \node[anchor=south west,inner sep=0] (img) 
        at (0,0) {\includegraphics[width=\linewidth]{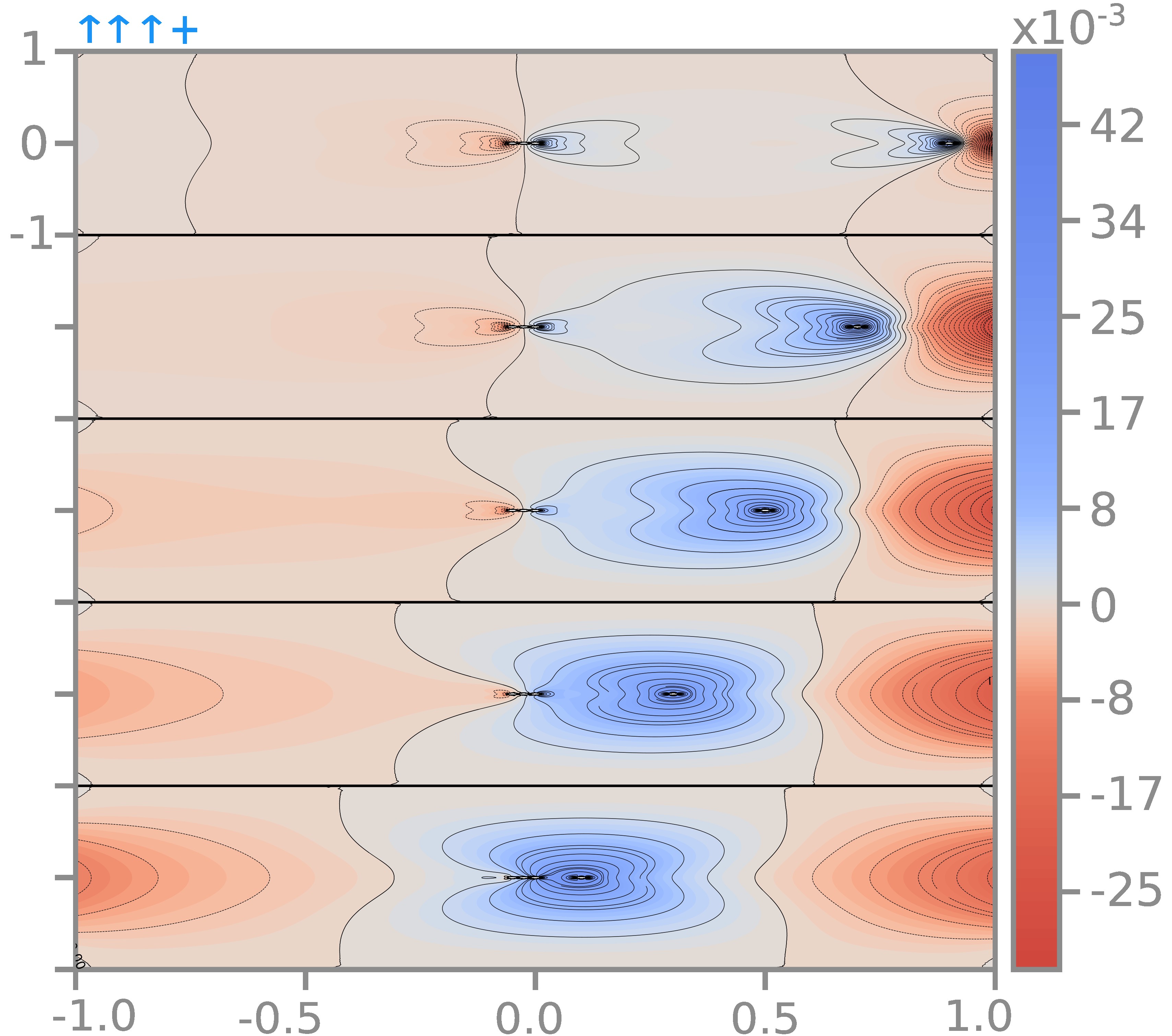}};
     \node[anchor=north] at ($(img.north)+(0,0.10)$) {$\sigma_{22}$};
    \node[anchor=north] at ($(img.south)+(0.35,0.5)$) {$x_1$};
\end{tikzpicture}
\end{minipage}

\begin{minipage}{0.47\textwidth}
\centering
\begin{tikzpicture}
    \node[anchor=south west,inner sep=0] (img) 
        at (0,0) {\includegraphics[width=\linewidth]{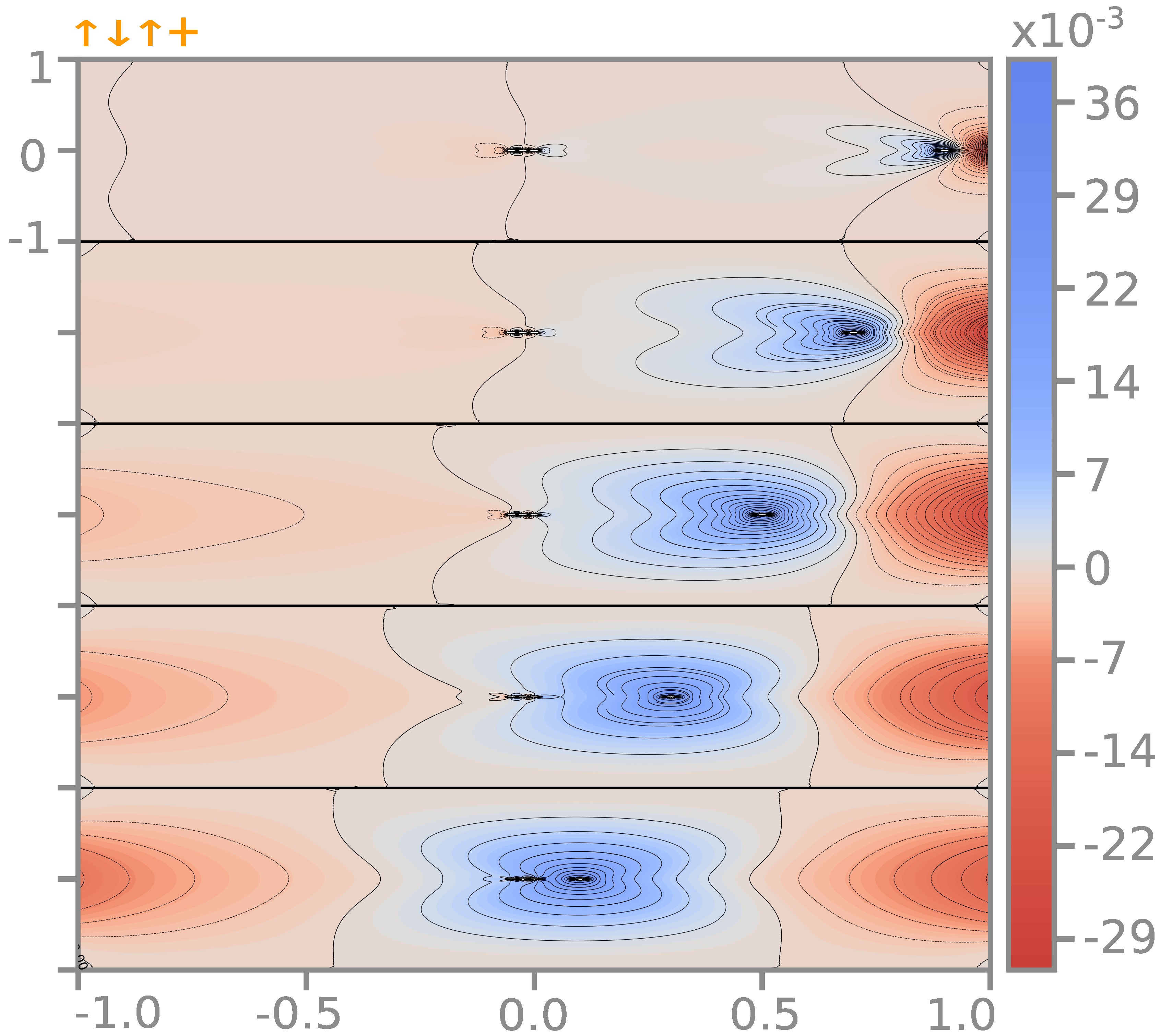}};
    \node[anchor=south] at ($(img.north)+(0,-0.4)$) {$\sigma_{22}$};
    \node[anchor=north] at ($(img.south)+(0.35,0.5)$) {$x_1$};
\end{tikzpicture}
\end{minipage}
\hfill
\begin{minipage}{0.47\textwidth}
\centering
\begin{tikzpicture}
    \node[anchor=south west,inner sep=0] (img) 
        at (0,0) {\includegraphics[width=\linewidth]{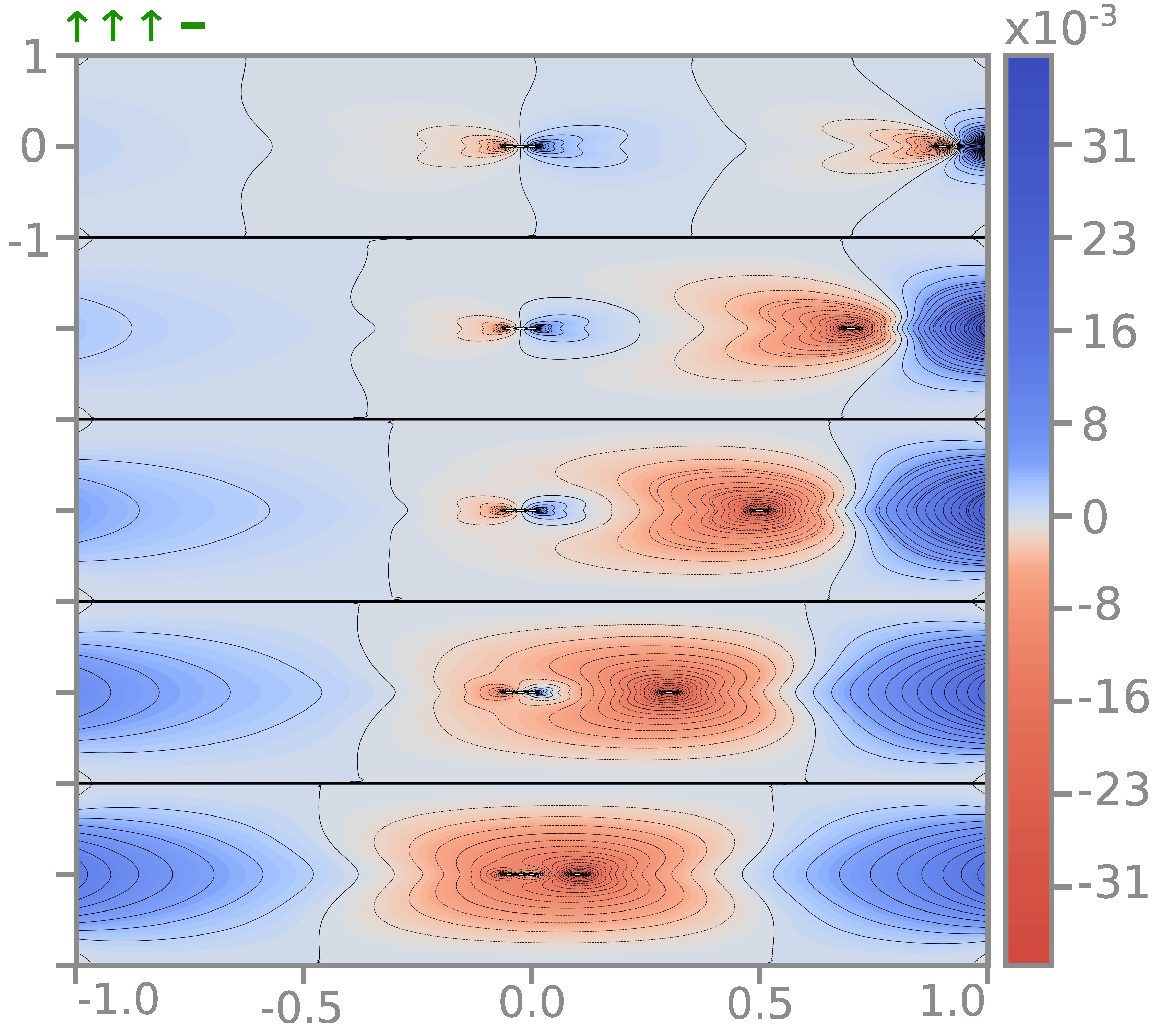}};
    \node[anchor=south] at ($(img.north)+(0,-0.4)$) {$\sigma_{22}$};
    \node[anchor=north] at ($(img.south)+(0.35,0.50)$) {$x_1$};
\end{tikzpicture}
\end{minipage}

\caption{Dislocation pile-up and isolated disclination. Top left: $\lVert \sigma \lVert_{L^2(\Omegaeps, \mathbb{R}^{2\times2})}$ as a function of $d$, the position of the wedge disclination on the $x_1$ axis. Circles represent numerically computed values, and squares (stars) indicate    maximum (local/global minimum) points. A solid line is included to guide the eye. Remaining panels:   contour plots of $\sigma_{22}$ for case (i) (top-right), case (ii) (bottom-left) and case (iii) (bottom-right). A mesh size of $\eta = 0.02$ is used, with additional refinement in the vicinity of the cores.  }
\label{fig:3D1d_combined}

\end{figure}
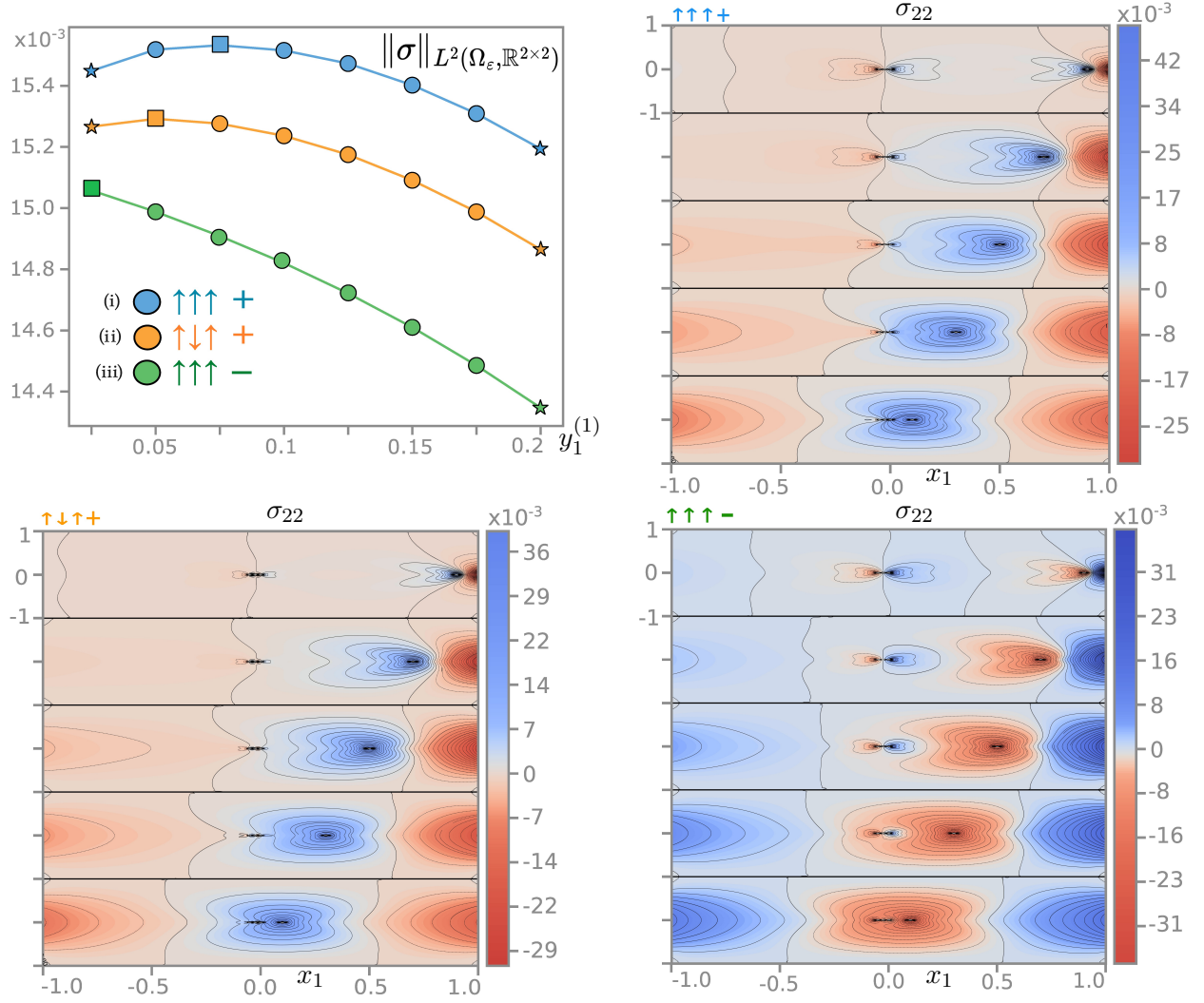
 
Three configurations are considered
(see Figure~$\ref{fig:3D1d_combined}$):
(i) the three edge dislocations have Burgers vectors aligned with the positive $x_2$--axis, and the wedge disclination has positive Frank angle $s = 0.1$ (blue curve); (ii) as in configuration (i), except the Burgers vector of the dislocation in the middle points downward (orange curve); 
(iii) as in  (i), except that the wedge disclination has a negative Frank angle $s = -0.1$ (green curve).

For the case (i), we observe (see top--left panel of Figure~$\ref{fig:3D1d_combined}$) $\lVert \sigma \rVert_{L^2(\Omegaeps, \mathbb{R}^{2\times2} )}$ attains its maximum when the wedge disclination is located at $d_{(i)} \approx 0.075$.
The parametric study reveals a critical disclination--stack separation establishing two energetic regimes: for smaller separations, that is, $x_1 = d < d_{(i)}$, the energy minimum corresponds to the disclination being closer to the dislocation stack, whereas for $x_1 = d > d_{(i)}$, it corresponds to a position closer to the boundary.
We observe a qualitatively similar behavior for the case (ii), although the critical distance $d_{(ii)}$ corresponding to the maximum is 
reduced to  $\approx 0.05$.
In case (iii), the stress behavior is monotone, and the mechanical system tends to separate the negative wedge disclination from the stack of edge dislocations.

The interaction between edge dislocations and the wedge disclination is localized and significant only when the distance 
among the edge dislocations is comparable to the disclination--stack separation, as illustrated by the stress field panels in Figure~\ref{fig:3D1d_combined}. 

Comparison of the stress curves for the three configurations shows that configuration (i) is the most mechanically stressed, while configuration (iii) is the least stressed. Inspection of the stress fields indicates that a negative disclination reduces the mechanical stress produced by the dislocation stack (see $\sigma_{22}$, bottom right), whereas a positive disclination increases it (see $\sigma_{22}$, top right). 
This result is consistent with the fact that a dislocation with an upward-pointing Burgers vector can be represented as a disclination dipole, where a negative wedge disclination precedes a positive one \cite{Eshelby66}.

\subsection{Stack of edge dislocations and dipole of wedge disclinations}
 
\begin{figure}[h!]
\centering
\includegraphics[width=0.6\linewidth]{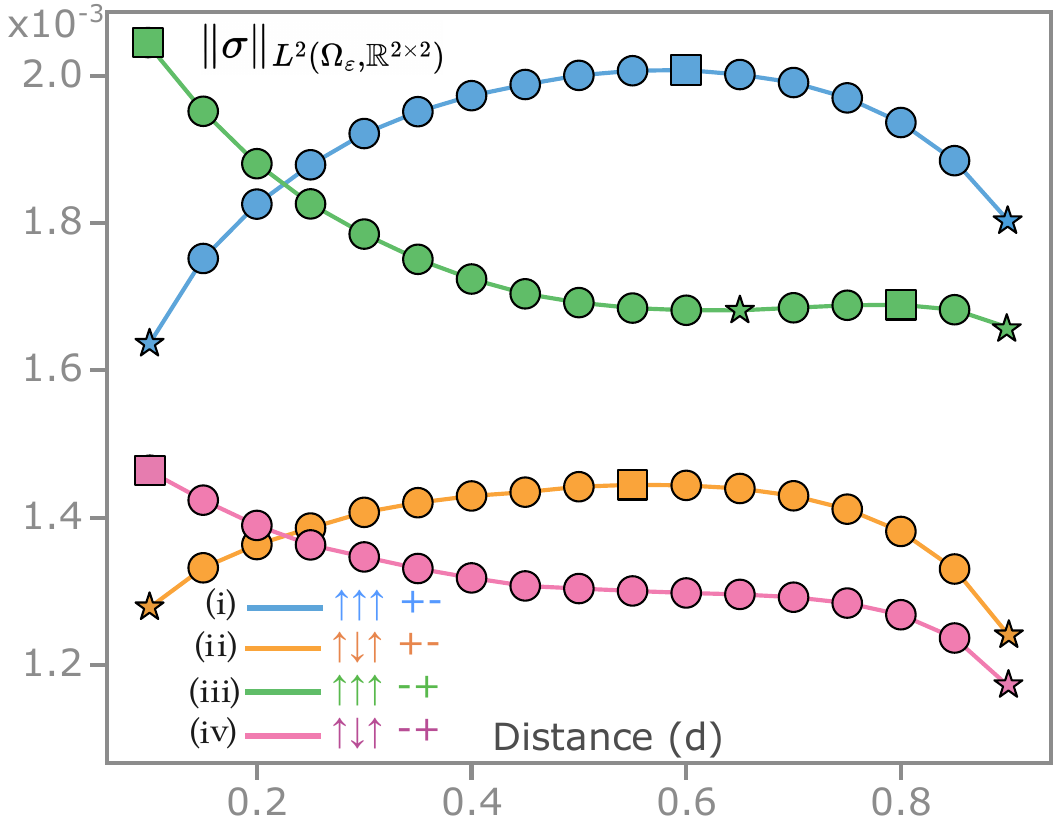}
\caption{$L^2$--norm of $\sigma$ as a function of the distance $d$ between the stack of edge dislocations and the dipole of wedge disclinations. Circles denote the numerically obtained values, while the continuous line is added to guide the eye. Square (star) symbols  represent  stationary points corresponding to local/global maxima (minima). A mesh size of $\eta = 0.02$ is used, with additional refinement in the vicinity of the cores.}
\label{fig:202509301231}
\end{figure}

We now consider a configuration consisting of a pile--up of edge dislocations interacting with a dipole of wedge disclinations, placed at varying distances.
As in the previous section, three edge dislocations are located at \(x_{1} = -0.05\), \(x_{1} = -0.025\), and \(x_{1} = 0\) within the square domain \(\Omega = (-1,1)^2\). We denote by \(d\) the distance between the dislocation pile--up and the closest disclination, and by \(h = 0.025\) the fixed separation between the two disclinations forming the dipole. The distance between the dislocation pile--up and the center of mass of the disclination dipole is therefore \(d + h/2\).

In contrast to the configuration studied in Section~\ref{sec:202509301227}, involving the interaction between a dislocation pile--up and a single disclination, which led to the emergence of strongly localized stress concentrations, the present setting exhibits long--range elastic interactions between the dislocation pile--up and the disclination dipole.

In Figure~\ref{fig:202509301231} we show the norm of the mechanical stress integrated over the associated perforated domain $\Omegaeps$ for several values of $d$ in the range $ [0.1, 0.9]$ and for four different configurations
defined by different orientations of Burgers vector and of the dipole.
Cases (i) and (ii) differ by an inversion of the orientation of the middle dislocation in the stack, while the disclination dipole is identically aligned in both configurations. In both cases, the stress curves exhibit very similar shapes, albeit with different magnitudes. Each curve displays a well--defined maximum at $d_{(i)} \approx 0.60$ and $d_{(ii)} \approx 0.55$, respectively. 
For small dipole--stack separation (below $d_{(i)}$), the energy minimum corresponds to the disclination dipole being closer to the dislocation stack, whereas for large separations (above $d_{(i)}$), it corresponds to a position closer to the boundary.
The behavior changes dramatically when the orientation of the disclination dipole is inverted, see cases (iii) and (iv).
The stress profiles are shallow and display a broad plateau,
with stress values varying within a narrow range over a large interval of the distance $d$.
In   case (iii), we observe a sequence of critical points consisting of a local minimum followed by a local maximum.
In the case (iv), the behavior of the stress curve is monotonic, and the stress decreases as the disclination dipole migrates toward the boundary.

\subsection{Disclination between a stack of dislocations}
\label{sec:202511151234}

\begin{figure}[h!]
\centering

\begin{minipage}{0.47\textwidth}
\centering
\includegraphics[width=\linewidth]{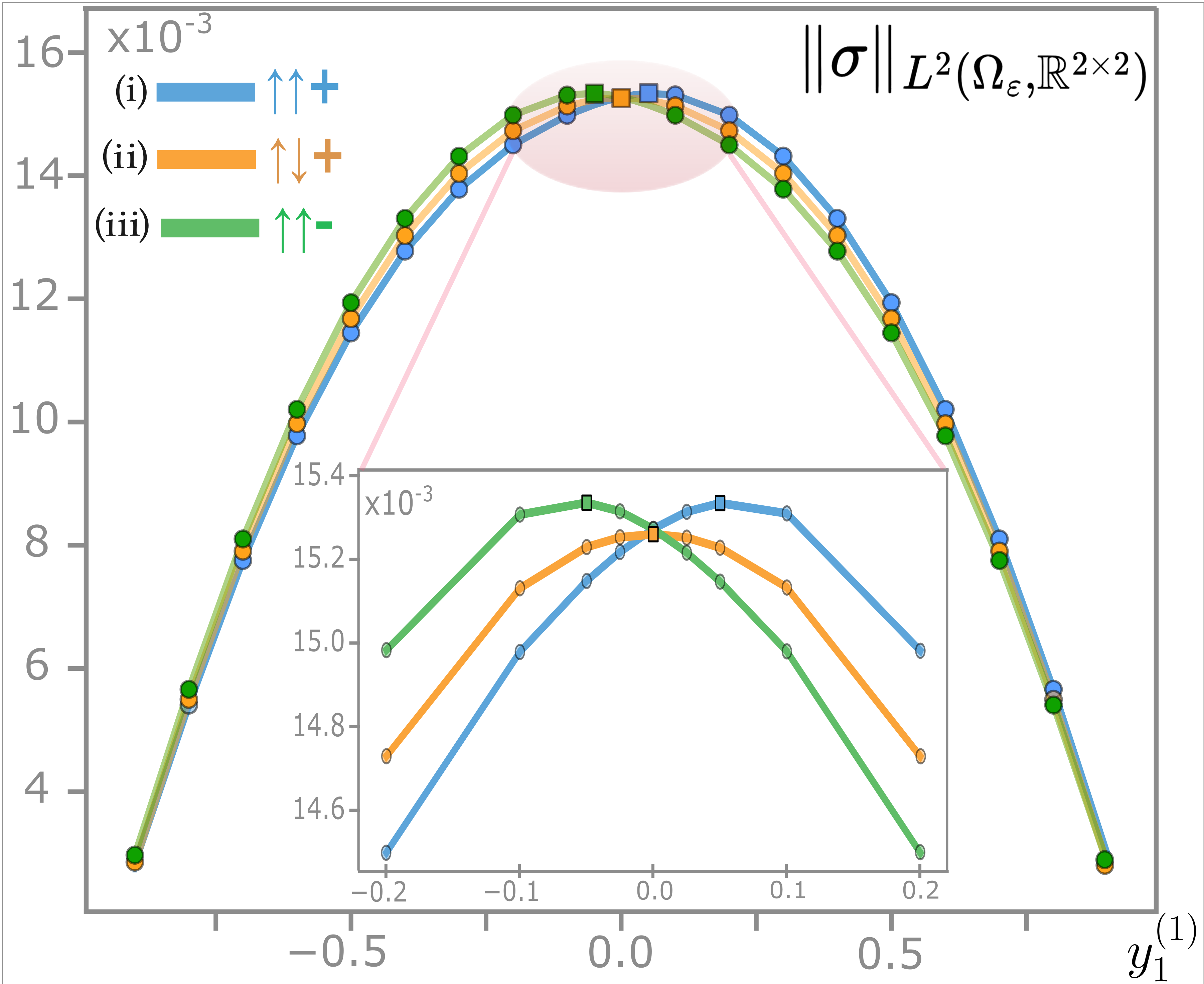} 
\end{minipage}
\hfill
\begin{minipage}{0.47\textwidth}
\centering
\includegraphics[width=\linewidth]{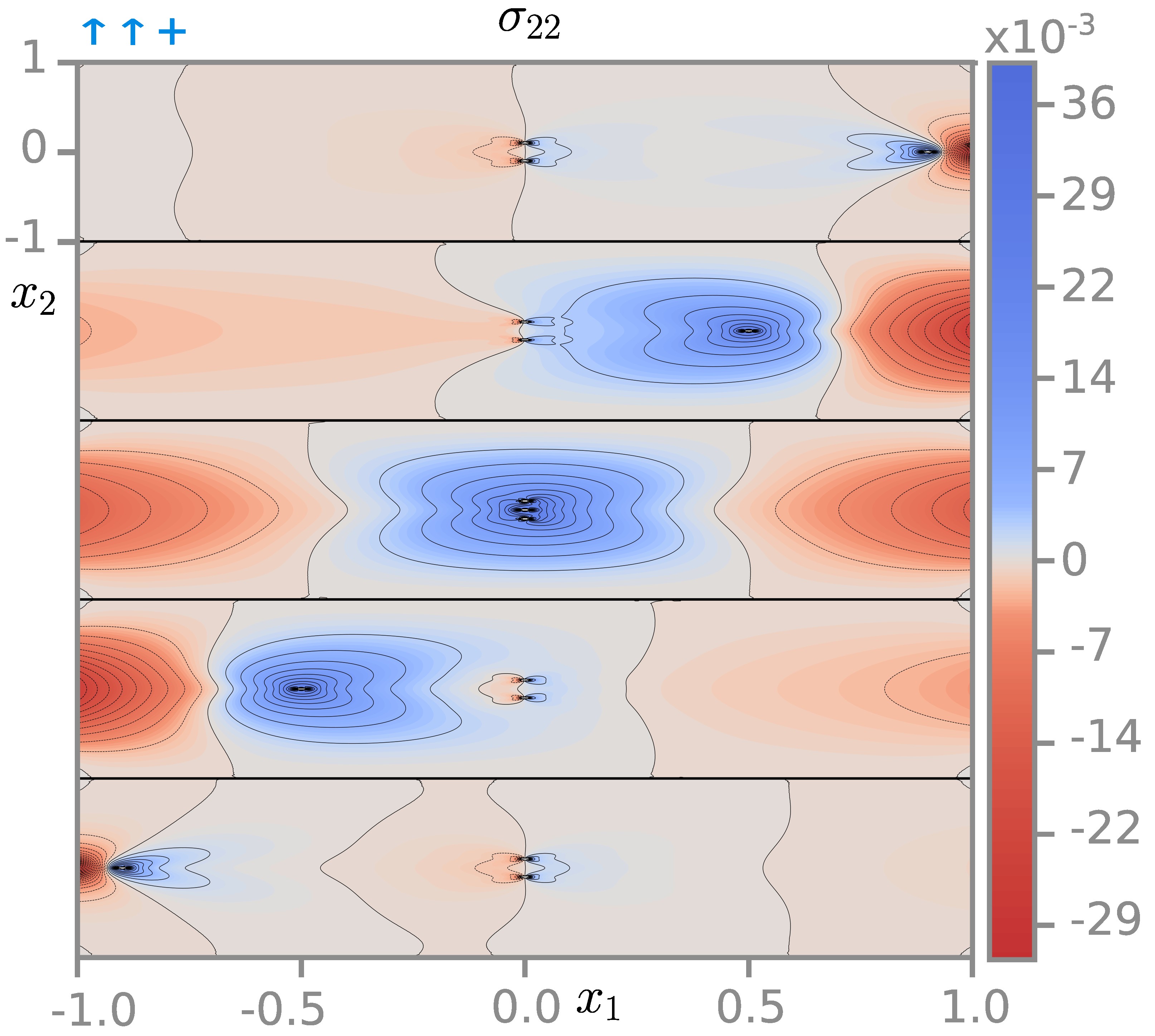}
\end{minipage}

\begin{minipage}{0.47\textwidth}
\centering
\includegraphics[width=\linewidth]{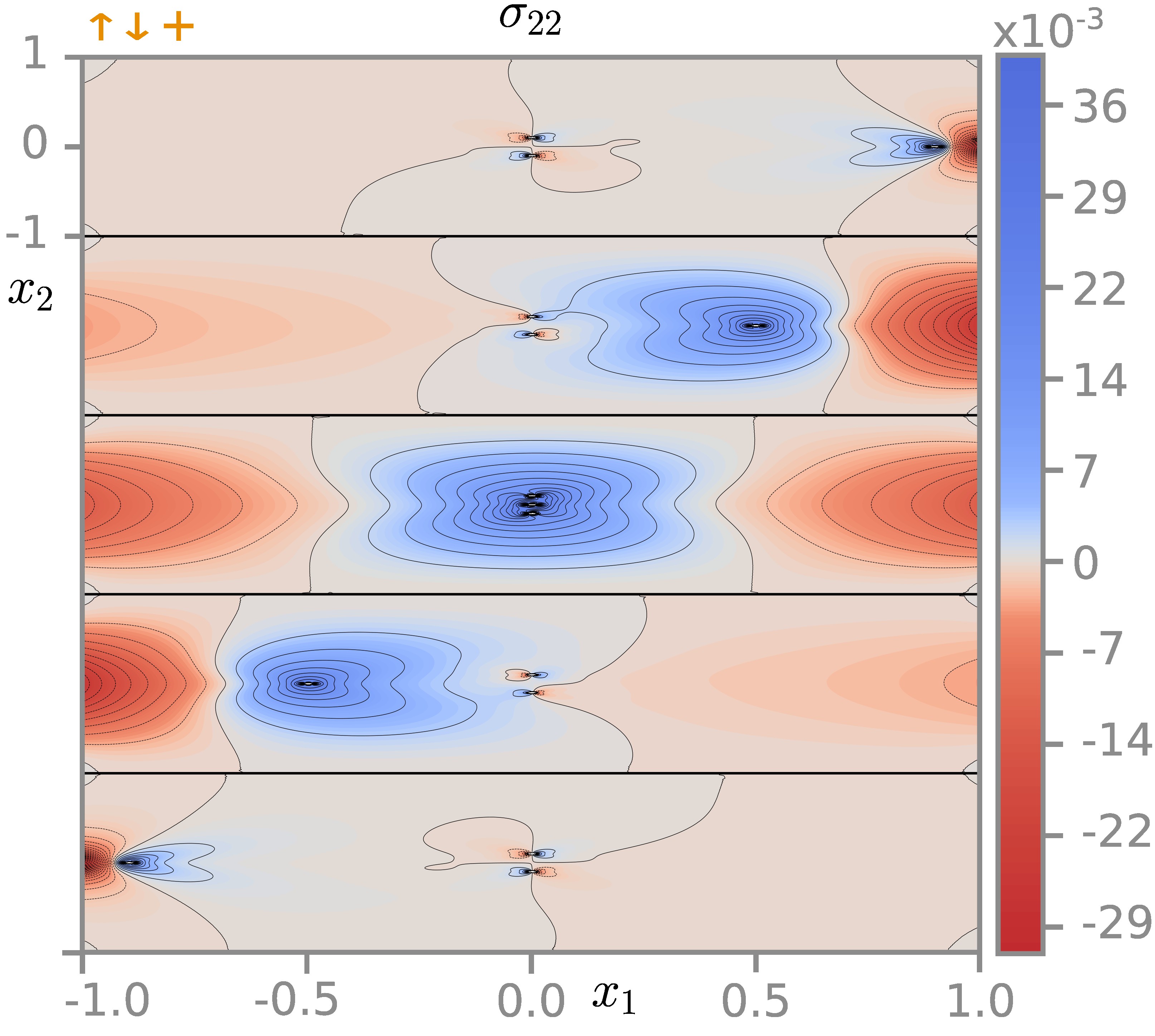}
\end{minipage}
\hfill
\begin{minipage}{0.47\textwidth}
\centering
\includegraphics[width=\linewidth]{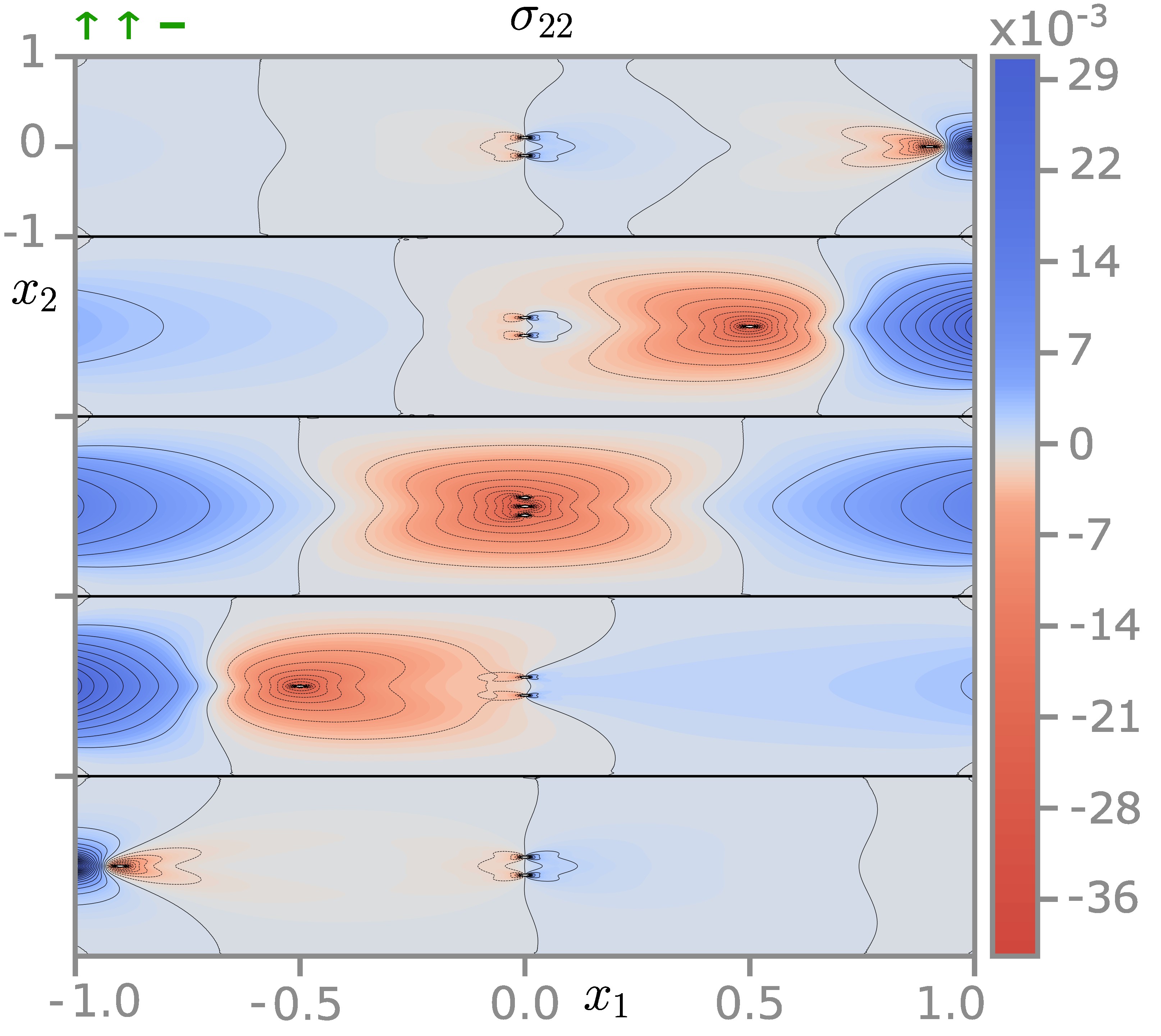}
\end{minipage}
\caption{Top--left panel: $L^2$--norm of $\sigma$ as a function of the disclination position along the $x_1$ axis. Circles mark the results of numerical experiments, while the solid line is included to guide the eye. Squares indicate maximum values. 
Other panels:   contour plots of $\sigma_{22}$ captured for five different positions of the wedge disclination, $y^{(1)}_1 = \pm 0.9$, $\pm 0.5$, $0.0$. A mesh size of $\eta = 0.02$ is used, with additional refinement in the vicinity of the cores.
}
\label{fig:202511151248}
\end{figure}

We consider a configuration with two fixed, vertically aligned dislocations and an isolated disclination whose horizontal position is varied between them.
As a result, the disclination is subject to a complex set of anisotropic and asymmetric elastic effects induced by the dislocation pair.   To analyze this behavior, we consider three distinct scenarios, each defined by the orientation of the Burgers vectors and the sign of the disclination.
The dislocations are placed along the vertical axis at the coordinates $(0,\,0.1)$ and $(0,\,-0.1)$ within the square domain $\Omega = (-1,\,1)^2$.

In Figure~\ref{fig:202511151248}, top left panel, we show the norm of the stress integrated over the associated $\Omegaeps$ as a function of the location of the wedge disclination for the three different scenarios studied.
We observe that the interaction effects are highly localized, see the inset in Figure~\ref{fig:202511151248} (top--left panel).
There exists a perfectly symmetric case (ii), in which the stress profile attains its maximum at the origin of the system, corresponding to the configuration where the disclination core is aligned with the cores of two dislocations having opposite Burgers vectors. In this case, the stress fields generated by the dislocations partially screen each other.
In the remaining two cases, where the Burgers vectors point in the same direction, the stress profiles exhibit an asymmetric behavior. Compared to these configurations, case (ii) appears overall less mechanically stressed, consistently with the partial screening induced by oppositely oriented Burgers vectors.

A common feature shared by all three configurations is that the stress magnitude decreases toward its minimum as the disclination is   closer to the boundary of the domain, further from the dislocation pair. 
On the boundary of $\Omega$, the normal component of the stress vanishes.
As expected, 
for $\sigma_{22}$ the zero trace values are achieved on the top and bottom boundaries.
However, when the disclination is located at the center of the domain or at an intermediate position, tangential stress components along the boundary   become significant.

\subsection{Interaction of disclinations with a cavity}

We study the effect of the interaction between two disclinations with opposite Frank angles, located on opposite sides of a rectangular domain containing a large cavity in the middle.
The aim is to investigate 
the screening effects of the stresses induced by the two disclinations, as influenced by the topology of the domain.
We consider the domain $\Omega = (-1,1) \times (-0.33, 0.33)$. 
The centers of the cores are placed at various positions along the horizontal axis midway through the rectangle, spanning from left to right, parametrized by $2d$, the mutual distance between the two disclinations.
A cavity is located at the center of the domain, and it is assumed to have no rotational or translational mismatch.
In our numerical framework,  this cavity is modeled as a core of varying radius~$R$, centered at $(0,0)$, with associated Frank angle $s = 0$ and Burgers vector $b = 0$.
To summarize, in this configuration we have
\[
\Omega_{\varepsilon} = \Omega \setminus \left\{ \overline{B}_R(0) \cup \overline{B}_{\varepsilon}\left(\pm d,0\right) \right\}.
\]
Results of the simulations are shown in Figure~\ref{fig:202511191032}, displaying the profiles and fields of the stresses for selected values of $R$ and $d$.
To enable comparison of the stress norms across configurations with different values of the inner radius, and consequently for different domains $\Omega_{\varepsilon}$\,, in Figure~\ref{fig:202511191032}--left we normalize the stress norm by dividing by $|\Omega_{\varepsilon}|$. 
Each curve exhibits a maximum corresponding to a critical distance between the two disclinations. For separations below this value, configurations with smaller distances are energetically favored, whereas for larger separations configurations with larger distances are energetically favored.

Remarkably, the maximum value of the normalized stresses increase as $R$ increases.
A qualitative interpretation of this phenomenon can be deduced from inspection of the stress component fields (see Figure~\ref{fig:202511191032}--right).
As $R$ increases, the screening effect of the two disclinations with opposite Frank angles progressively weakens. In the final case, for $R = 0.3$, the system behaves   as if there were two isolated disclinations in two disconnected domains.

Finally, we observe that, across different values of $R$, the stress curves exhibit the same trend for large values of $d$, that is, when the disclinations are located close to the boundaries of the domain.

\begin{figure}[h]
    \centering
    \begin{subfigure}[t]{0.5\linewidth}
        \centering
        \includegraphics[width=\linewidth]{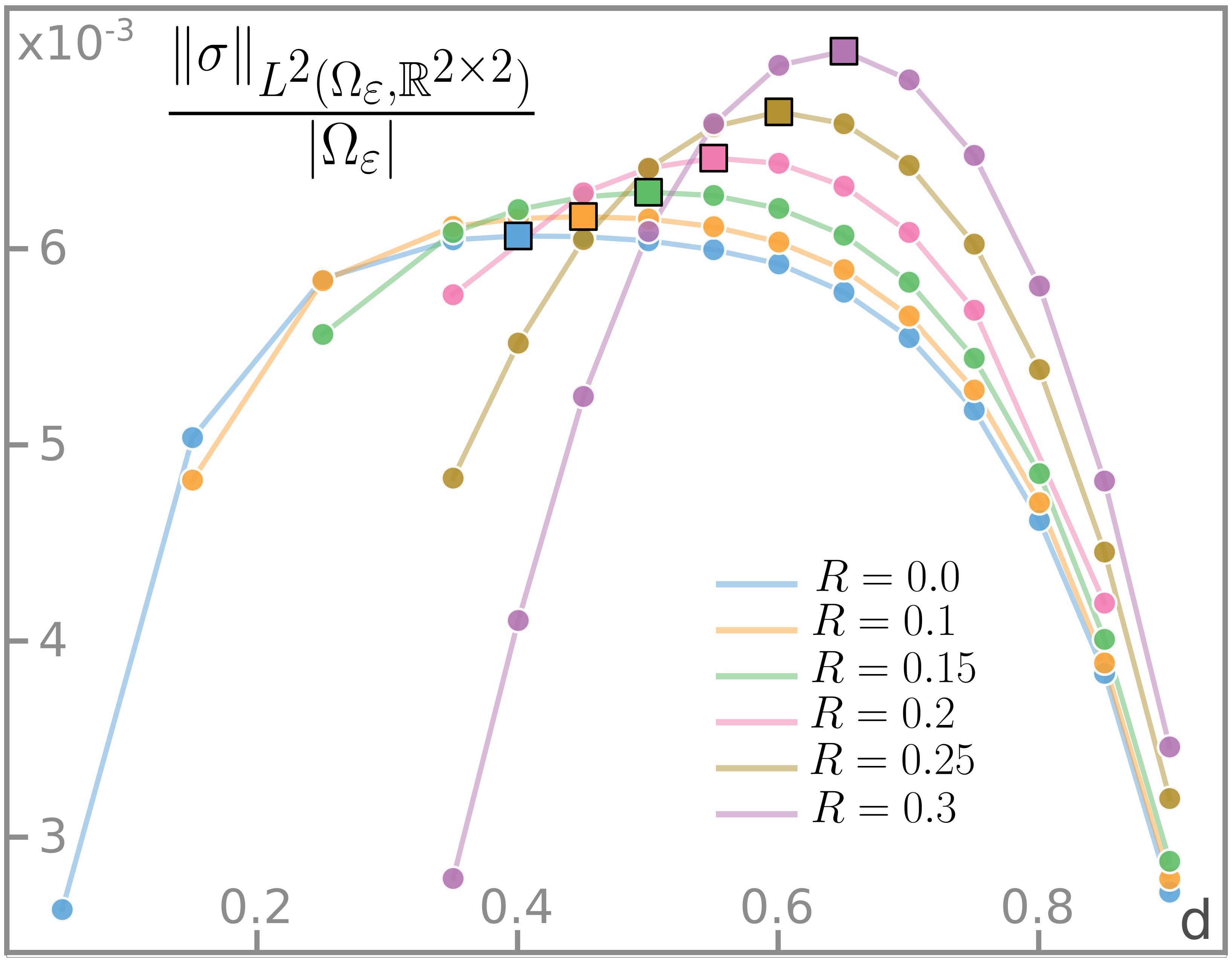}
    \end{subfigure}
\qquad
    \begin{subfigure}[t]{0.25\linewidth}
        \centering
        \includegraphics[width=\linewidth]{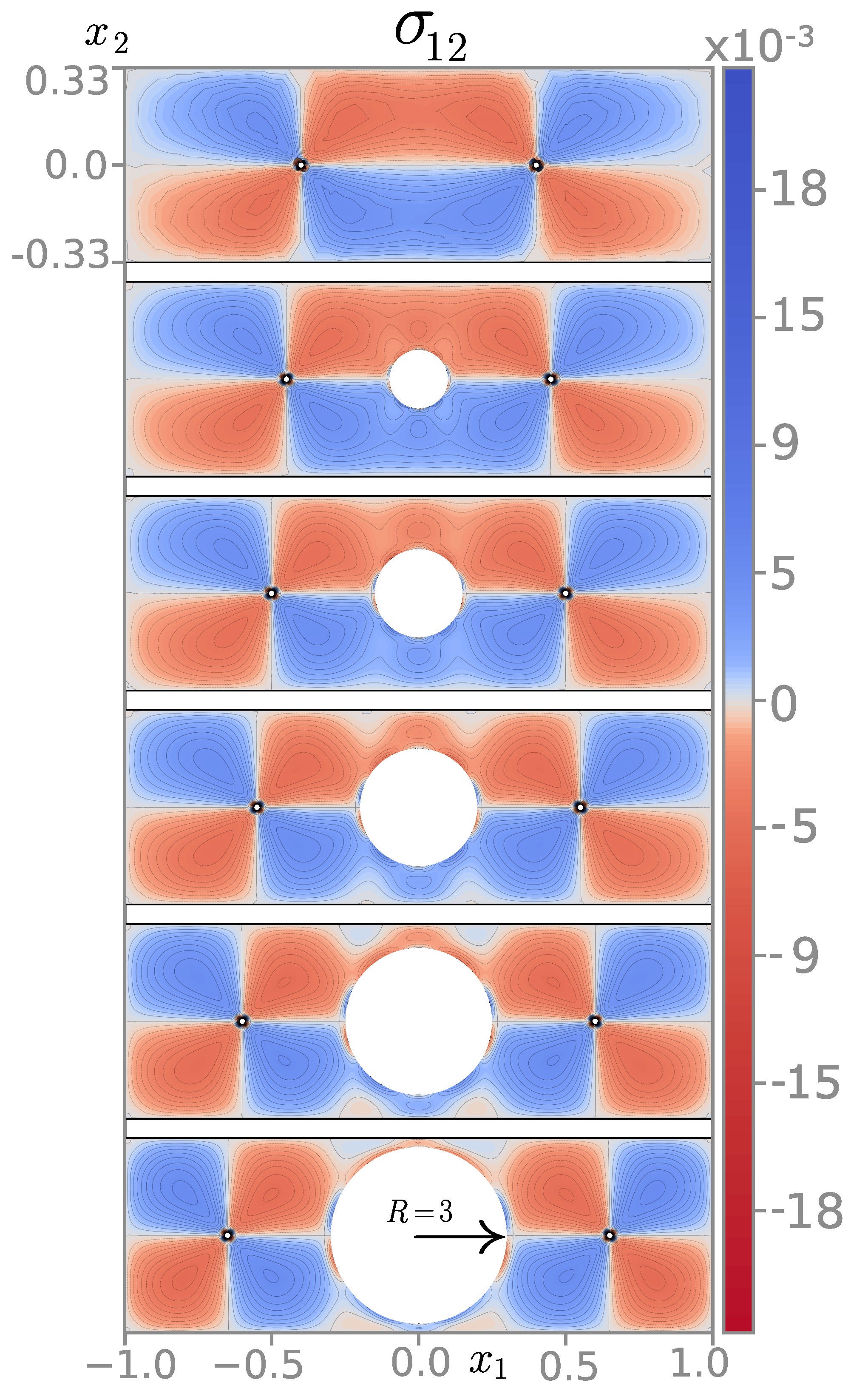}
    \end{subfigure}
    \caption{Left panel: normalized stress norm ($\lVert \sigma \rVert_{L^2(\Omega_{\varepsilon},\mathbb{R}^{2\times2})} / |\Omega_{\varepsilon}|$) for different values of the mutual distance between the two wedge disclinations and for different values of $R$. 
Squares indicate the values of \( d \) that maximize the stress norm for each \( R \).
For \( d > 0.6 \), larger values of \( R \) result in higher normalized stress.
Right panel: stress field $\sigma_{12}$ displayed for the value of   $d$ that maximizes the stress norm, indicated with squares in the left panel. 
The screening effects in the stresses, present for small values of $R$, tend to vanish as $R$ increases. A mesh size of $\eta = 0.02$ is used, with additional refinement in the vicinity of the cores.} 
    \label{fig:202511191032}
\end{figure}

\section*{Conclusion}

We developed a numerical model for resolving dislocation--disclination interactions in domains with Lipschitz-regular boundaries. We implement the finite element method for Dirichlet problems involving the bilaplacian, where dislocations and disclinations are represented through incompatibilities prescribed on $\varepsilon$-scale cores following the core-radius regularization approach. The framework is general: by solving a collection of cell problems associated with fundamental solutions for selected unitary boundary conditions, the solution for any configuration of dislocations and disclinations can subsequently be obtained through linear superposition at the postprocessing stage. 

The key advantage is that, once the cell problems are solved for a given geometry and topology of the problem, interactions between defects of any intensity, orientation, and nature (dislocations/disclinations) follow from simple postprocessing.

We validated the model on a set of configurations exhibiting nontrivial dislocation–disclination interactions, highlighting how translational and rotational incompatibilities influence the stress states of the system.
We have examined the energy response of these configurations and observed a dependence of energy minima and maxima on defects arrangement and geometry. This opens the door to the investigation of the driving forces, stability, and spatiotemporal evolution of coupled dislocation-disclination systems.

Future developments will include extensions toward nonlinear models to study realistic metallic materials under large deformations, with the goal of capturing the emergence of plasticity from the interplay of nanoscale defects and geometric incompatibilities. This includes the observation of hardening and softening behavior and the computation of full stress-strain curves. Another direction is the extension of the present two--dimensional planar strain model to plate theories, such as the von Kármán model, which requires coupling the Airy stress function with the out-of-plane deflection variable (see \cite{FABBRINI25}). 
Further extensions include the systematic study of plasticity, damage, fracture, and fragmentation, and their mutual interactions. Finally, the model can serve as a benchmark for comparison with ab-initio simulations investigating the coupled dynamics of dislocations and disclinations.

\paragraph{Conflict of interest}
The authors declare no conflict of interest.

\paragraph{Acknowledgements}
This work was started while EF was a doctoral student at the Graduate School of Mathematics, Kyushu University, with support from JST SPRING (Grant Number JPMJSP2136). It was completed during his postdoctoral appointment at the Graduate School of Science, Kyoto University, where he is partially supported by JST Moonshot R\&D (Grant Number JPMJMS2021).
PC’s work is supported by JSPS KAKENHI Grant--in--Aid for Scientific Research (C) JP24K06797 and  by JST A--STEP (Grant Number JPMJTR24T6). 
MM acknowledges partial support from the MUR grant Geometric Analytic Methods for PDEs and Applications (2022SLTHCE cup E53D23005880006). This manuscript reflects only the authors’ views and opinions and the Italian Ministry cannot be considered responsible for them.

\noindent PC holds an honorary appointment at La Trobe University. 
PC and MM are members of the Gruppo Nazionale per l’Analisi Matematica, la Probabilità e le loro Applicazioni (GNAMPA) of the Istituto Nazionale di Alta Matematica (INdAM). 
MM thanks the Institute of Mathematics for Industry, an International Joint Usage and Research Center located in Kyushu University; PC and EF thank the Department of Mathematical Sciences ``G.~L.~Lagrange'' of Politecnico di Torino where part of the work contained in this paper was carried out.

\appendix
\section{Technical results}
In this appendix we collect a few technical results that are needed in the proof of Theorem~\ref{2502282100}.

\subsection{Equivalence of boundary conditions for piecewise-$C^2$ Lipschitz domains}

Here we show that if $A$ is a piecewise--$C^2$ domain according to Definition~\ref{def_Omega} and $v \in C^2(\overline{A})$, then the boundary condition $\nabla^2v\,t = 0$ on $\partial A$ is equivalent to requiring that $v|_{\Gamma}$ be the trace of an affine function on every connected component $\Gamma$ of $\partial A$. 

\begin{proposition}\label{prop202510111410}
Let $A \subset \mathbb{R}^2$ be an open and bounded set as in Definition \ref{def_Omega} and let $v\in C^2(\overline{A})$. 
Then for every connected component $\Gamma_q$ ($q=1,\ldots,Q$, $Q\geq1$) of $\partial A$ we have that
\begin{equation}
\nabla^2v t= 0\quad \text{on $\Gamma_q$} 
\qquad\Leftrightarrow \qquad 
v= a_q, \quad \partial_n v= \partial_n a_q \quad\text{on $\Gamma_q$,}
\end{equation}
for some affine functions $a_q$\,, $q=1,\ldots,Q$.
\end{proposition}
\begin{proof}
Let $\Gamma_q$ be a connected component of $\partial A$.
On each curve $\Gamma^\ell_q$ ($\ell=1,\ldots,L(Q)$), apply \cite[Proposition A.2]{Cesana2024a} to obtain that there exist $L(q)$ affine functions (and therefore $3L(q)$ real parameters)
$$\Gamma^\ell_q\ni x^\ell_q\mapsto a_q^\ell(x_q^\ell)=a_{q,1}^\ell x_1^\ell+a_{q,2}^\ell x_2^\ell +a_{q,0}^\ell$$ 
such that $v|_{\Gamma^\ell_q}(x^\ell_q)=a_q^\ell(x_q^\ell)$, for every $\ell=1,\ldots,L(q)$. 

Let us denote by $g_{q,0}$ and $g_{q,1}$ the traces of $v$ and of its normal derivative $\partial_n v$, respectively, on $\Gamma_q$\,.
Therefore, we can apply \cite[Theorem 3]{GK2000} which yields that the condition
\begin{equation}\label{GK_condition}(\partial_t g_{q,0})n - g_{q,1} t \in H^{1/2}(\Gamma_q)
\end{equation}
must be satisfied.
Besides implying continuity, \eqref{GK_condition} implies that there are no jumps at the corners $P_q^\ell$ of the boundary.

Eventually, this implies that the trace $g_{q,0}$ at the boundary is \emph{only one} affine function, \emph{i.e.}, $a_q^\ell=a_q^1$ for every $\ell=2,\ldots,L(q)$. 
To see this, we impose the continuity of \eqref{GK_condition} and of $g_{q,0}$ at the $P_q^\ell$'s.
By our positions, given $\ell$, the point $P_q^\ell$ is the end point of $\Gamma_q^\ell$ and the starting point of $\Gamma_q^{\ell+1}$.
Let us denote by $t_q^{\ell}=(t_{q,1}^\ell,t_{q,2}^\ell)^\top$ and by $n_q^\ell=(n_{q,1}^\ell,n_{q,2}^\ell)^\top=(t_{q,2}^\ell,-t_{q,1}^\ell)^\top$ the tangent and normal vectors, respectively.
The continuity of \eqref{GK_condition} across $P_q^\ell$ reads
$$(\langle \nabla a_q^\ell,t_q^\ell\rangle)n_q^\ell-(\langle\nabla a_q^\ell,n_q^\ell\rangle)t_q^\ell = (\langle \nabla a_q^{\ell+1},t_q^{\ell+1}\rangle)n_q^{\ell+1}-(\langle\nabla a_q^{\ell+1},n_q^{\ell+1}\rangle)t_q^{\ell+1}\,$$
which, in components, reads
\begin{equation}\label{cont_grad}
\begin{pmatrix}
a_{q,2}^\ell\\
-a_{q,1}^\ell
\end{pmatrix} = 
\begin{pmatrix}
a_{q,2}^{\ell+1}\\
-a_{q,1}^{\ell+1}
\end{pmatrix},\qquad\text{for every $\ell=1,\ldots,L(q)$,}
\end{equation}
where we have used that $\lVert t_q^\ell\rVert=1$ for every $\ell=1,\ldots,L(q)$; this fixes $2L(q)$ constraints.
The continuity equations of $g_{q,0}$ at each $P_q^\ell=(X_{q,1}^\ell,X_{q,2}^\ell)$ reads
$$a_{q,1}^{\ell} X_{q,1}^{\ell} + a_{q,2}^{\ell} X_{q,2}^{\ell} + a_{q,0}^{\ell} = a_{q,1}^{\ell+1} X_{q,1}^{\ell} + a_{q,2}^{\ell+1} X_{q,2}^{\ell} + a_{q,0}^{\ell+1},\qquad \text{for every $\ell=1,\ldots,L(q)$,}$$
which fixes $L(q)$ constraints and, by virtue of \eqref{cont_grad}, implies that $a_{q,0}^\ell=a_{q,0}^1$ for every $\ell=2,\ldots,L(q)$.

The special case $L(q)=1$ can be treated by adding a fictitious point $P_q^2$ and repeating the argument above with $L(q)=2$.
\end{proof}

\subsection{A symmetry property of the Monge-Amp\`{e}re operator} 
In this section we recall the definition of the Monge--Amp\`{e}re operator and extend   \cite[Lemma C.2]{CFM2025} to the case of $\partial \Omega$ Lipschitz and functions with $H^2$ regularity.

\begin{definition}[Monge--Amp\`{e}re operator]\label{sec:24241126922}
Let $\Omega\subset\mathbb{R}^2$ be an open set, for any $\xi,\eta \in H^2(\Omega)$, the Monge--Amp\`{e}re operator is defined as
\begin{equation}
\label{eq:ma}
[\xi, \eta](x)\coloneqq \cof(\nabla^2 \xi(x)) : \nabla^2 \eta(x) \qquad \text{for a.e.~$x \in \Omega$.}
\end{equation}
\end{definition}

We now present Lemma \ref{2909251140} establishing an important property for the following trilinear form
\begin{equation}\label{eq_20251114}
H^2(\Omega) \times H^2(\Omega) \times H^2(\Omega) \ni (\psi, \eta, \chi) \mapsto \int_{\Omega} [\psi, \eta](x)\chi(x)\,\ud x.
\end{equation}

\begin{lemma}\label{2909251140}
Let $\Omega$ be as in Definition~\ref{def_Omega} and let $\Omega_\varepsilon$ be defined as in \eqref{eq_Omega_eps}.
Let $\eta, \chi, \psi \in H^2(\Omega_\varepsilon)$ with $\psi$ such that $\psi = \partial_n \psi = 0$ on $\partial \Omega$, and $\psi$ has affine trace on each $\partial B_{\varepsilon}^i$\,, for $i \in \{1, \hdots, N\}$. 
Then
\begin{equation}\label{eq_C3}
\int_{\Omegaeps} [\psi, \eta](x)\chi(x)\,\ud x= 
\int_{\Omegaeps} [\psi, \chi](x)\eta(x)\,\ud x.
\end{equation}
\end{lemma}
\begin{proof} 
Consider a sequence of functions $\psi^m \in C^{\infty}_c(\Omega)$ approximating $\psi$ in the $H^2(\Omega_\varepsilon)$--norm and such that 
$\psi^m$ are affine on every $\partial B_{\varepsilon}^{i}$\,. 
By the definition of the Monge--Amp\`{e}re operator \eqref{eq:ma} and using integration by parts, one obtains
\begin{equation}
\begin{aligned}
\int_{\Omega_\varepsilon} [\psi^{m}, \eta] \, \chi \, \ud x & = \int_{\Omega_\varepsilon} \chi \, \cof(\nabla^2 \psi^{m}) : \nabla^2 \eta \, \ud x \\
&= \int_{\partial \Omega_\varepsilon} \chi \, \langle \cof(\nabla^2 \psi^{m})n, \nabla \eta\rangle \,\ud\Huno - \int_{\Omega_\varepsilon} \big\langle \Div \big(\chi \, \cof(\nabla^2 \psi^{m})\big), \nabla \eta\big\rangle \,\ud x  \\
&= \sum_{i=1}^N \int_{\partial B_{\varepsilon}(\xi^i)}  \chi \, \langle \cof(\nabla^2 \psi^{m})n, \nabla \eta \rangle \, \ud \Huno - \int_{\Omega_\varepsilon} \big\langle \Div \big(\chi \, \cof(\nabla^2 \psi^{m})\big), \nabla \eta\big\rangle \,\ud x .
\end{aligned}
\end{equation}
The integrals under the summation symbol are individually zero: this follows from \cite[Proposition~A.2]{Cesana2024a} applied to $\psi^{m}$, using the identity $\cof(M)=\Pi^\top M\Pi$ for any $M\in\mathbb{R}^{2\times2}_{\sym}$ together with the relation $\Pi n=-t$.
We are therefore left with
\begin{equation}
\begin{aligned}
\label{eq:202509291410}
\int_{\Omega_\varepsilon} [\psi^{m}, \eta] \, \chi \, \ud x
&= -\int_{\Omega_\varepsilon} \big\langle \Div \big(\chi \, \cof(\nabla^2 \psi^{m})\big), \nabla \eta\big\rangle \,\ud x \\
&= - \int_{\Omega_\varepsilon} \big\langle  \cof(\nabla^2 \psi^{m}) \nabla \chi, \nabla \eta\big\rangle \,\ud x - \int_{\Omega_\varepsilon} \big\langle  \Div (\cof(\nabla^2 \psi^{m})), \nabla \eta\big\rangle \chi \,\ud x \\
&= - \int_{\Omega_\varepsilon} \big\langle  \cof(\nabla^2 \psi^{m}) \nabla \chi, \nabla \eta\big\rangle \,\ud x ,
\end{aligned}
\end{equation}
since a straightforward computation shows that $\Div(\cof(\nabla^2 \psi^{m})) = 0$ pointwise in $\Omega_\varepsilon$\,.  
The last expression in \eqref{eq:202509291410} is symmetric in $\eta$ and $\chi$, and so is the first one. We have thus shown that 
\begin{equation}
\label{eq:202509291411}
\int_{\Omega_\varepsilon} [\psi^{m}, \eta] \, \chi \, \ud x
= \int_{\Omega_\varepsilon} [\psi^{m}, \chi] \, \eta \, \ud x .
\end{equation}
Finally, passing to the limit as $m \to \infty$ yields   \eqref{eq_C3}.
\end{proof}

\bibliographystyle{elsarticle-num}
\bibliography{refsPatrick.bib}

@article{AlicandroDeLucaPalombaroPonsiglione2025,
	author = {Alicandro, Roberto and De Luca, Lucia and Palombaro, Mariapia and Ponsiglione, Marcello},
	doi = {10.1515/acv-2023-0053},
	fjournal = {Advances in Calculus of Variations},
	issn = {1864-8258,1864-8266},
	journal = {Adv. Calc. Var.},
	mrclass = {74C05 (49J45 70G75 74B20)},
	mrnumber = {4845983},
	number = {1},
	pages = {1--23},
	title = {{$\Gamma $}-convergence analysis of the nonlinear self-energy induced by edge dislocations in semi-discrete and discrete models in two dimensions},
	volume = {18},
	year = {2025}}

@article{FABBRINI25,
	author = {Edoardo Fabbrini and Andr{\'e}s A. {Le{\'o}n Baldelli} and Pierluigi Cesana},
	doi = {https://doi.org/10.1016/j.apm.2025.116234},
	issn = {0307-904X},
	journal = {Applied Mathematical Modelling},
	pages = {116234},
	title = {Kinematically incompatible {F}{\"o}ppl--von {K}{\'a}rm{\'a}n plates: Analysis and numerics},
	volume = {148},
	year = {2025}}

@article{ZHANG18,
	author = {Chiqun Zhang and Amit Acharya and Saurabh Puri},
	doi = {https://doi.org/10.1016/j.jmps.2018.02.004},
	issn = {0022-5096},
	journal = {Journal of the Mechanics and Physics of Solids},
	pages = {258-302},
	title = {Finite element approximation of the fields of bulk and interfacial line defects},
	volume = {114},
	year = {2018}}

@article{Xu97,
	author = {G. Xu and A. S. Argon and M. Ortiz},
	doi = {10.1080/01418619708205146},
	journal = {Philosophical Magazine A},
	number = {2},
	pages = {341--367},
	publisher = {Taylor \& Francis},
	title = {Critical configurations for dislocation nucleation from crack tips},
	volume = {75},
	year = {1997}}

@article{Xu95,
	author = {G. Xu and A. S. Argon and M. Ortiz},
	doi = {10.1080/01418619508239933},
	journal = {Philosophical Magazine A},
	number = {2},
	pages = {415--451},
	publisher = {Taylor \& Francis},
	title = {Nucleation of dislocations from crack tips under mixed modes of loading: Implications for brittle against ductile behaviour of crystals},
	volume = {72},
	year = {1995}}

@book{Sadd25,
	author = {Martin H. Sadd},
	edition = {5},
	isbn = {9780443132452},
	publisher = {Elsevier},
	title = {Elasticity: Theory, Applications, and Numerics},
	year = {2025},
    doi = {10.1016/C2022-0-02029-8}}

@article{Zdzisaw99,
	author = {Zdzis{\l}aw Wi{\k{e}}ckowski and Sang Kook Youn and Bong Soo Moon},
	doi = {10.1002/(SICI)1097-0207(19990410)44:10<1505::AID-NME555>3.0.CO;2-G},
	journal = {International Journal for Numerical Methods in Engineering},
	number = {10},
	pages = {1505--1525},
	publisher = {John Wiley & Sons},
	title = {Stress‐based finite element analysis of plane plasticity problems},
	volume = {44},
	year = {1999}}

@article{Kundin11,
	author = {Julia Kundin and Heike Emmerich and Johannes Zimmer},
	doi = {10.1080/14786435.2010.485587},
	journal = {Philosophical Magazine},
	number = {1},
	pages = {97--121},
	publisher = {Taylor \& Francis},
	title = {Mathematical concepts for the micromechanical modelling of dislocation dynamics with a phase-field approach},
	volume = {91},
	year = {2011}}

@book{cai06,
	author = {Bulatov, Vasily and Cai, Wei},
	doi = {10.1093/oso/9780198526148.001.0001},
	isbn = {9780198526148},
	month = {11},
	publisher = {Oxford University Press},
	title = {Computer Simulations of Dislocations},
	year = {2006}}

@article{WANG10,
	author = {Yunzhi Wang and Ju Li},
	doi = {https://doi.org/10.1016/j.actamat.2009.10.041},
	issn = {1359-6454},
	journal = {Acta Materialia},
	number = {4},
	pages = {1212-1235},
	title = {Phase field modeling of defects and deformation},
	volume = {58},
	year = {2010}}

@article{WANG01,
	author = {Y.U. Wang and Y.M. Jin and A.M. Cuiti{\~n}o and A.G. Khachaturyan},
	doi = {https://doi.org/10.1016/S1359-6454(01)00075-1},
	issn = {1359-6454},
	journal = {Acta Materialia},
	number = {10},
	pages = {1847-1857},
	title = {Nanoscale phase field microelasticity theory of dislocations: model and 3D simulations},
	volume = {49},
	year = {2001}}

@article{RODNEY03,
	author = {D. Rodney and Y. {Le Bouar} and A. Finel},
	doi = {https://doi.org/10.1016/S1359-6454(01)00379-2},
	issn = {1359-6454},
	journal = {Acta Materialia},
	number = {1},
	pages = {17-30},
	title = {Phase field methods and dislocations},
	volume = {51},
	year = {2003}}

@article{CGMP2025,
	author = {Cesana, Pierluigi and Grillo, Alfio and Morandotti, Marco and Pastore, Andrea},
	doi = {10.1137/24M1688096},
	journal = {SIAM Journal on Applied Mathematics},
	number = {4},
	pages = {1361-1386},
	title = {Dissipative {D}ynamics of {V}olterra {D}isclinations},
	volume = {85},
	year = {2025}}

@article{FRESSENGEAS2020104092,
	author = {Claude Fressengeas and Xiaoyu Sun},
	doi = {https://doi.org/10.1016/j.jmps.2020.104092},
	issn = {0022-5096},
	journal = {Journal of the Mechanics and Physics of Solids},
	pages = {104092},
	title = {On the theory of dislocation and generalized disclination fields and its application to straight and stepped symmetrical tilt boundaries},
	volume = {143},
	year = {2020}}

@article{Becker21,
	author = {M. Nguyen-Hoang and W. Becker},
	doi = {https://doi.org/10.1016/j.ijsolstr.2021.03.010},
	issn = {0020-7683},
	journal = {International Journal of Solids and Structures},
	pages = {111023},
	title = {Stress analysis of finite dimensions bolted joints using the Airy stress function},
	volume = {224},
	year = {2021}}

@article{Cesana2024a,
	author = {Cesana, Pierluigi and De Luca, Lucia and Morandotti, Marco},
	journal = {SIAM Journal on Mathematical Analysis},
	month = {01},
	pages = {79-136},
	title = {Semidiscrete Modeling of Systems of Wedge Disclinations and Edge Dislocations via the Airy Stress Function Method},
	volume = {56},
	year = {2024},
    DOI = {10.1137/22M1523443}}

@article{RRK2018,
	author = {A. E. Romanov and M. A. Rozhkov and A. L. Kolesnikova},
	journal = {Letters on Materials},
	number = {4},
	pages = {384-400},
	title = {Disclinations in polycrystalline graphene and pseudo-graphenes. {R}eview},
	volume = {8},
	year = {2018},
    doi = {10.22226/2410-3535-2018-4-384-400}}

@article{KVV2016,
	author = {M. Kabir and K. J. Van Vliet},
	journal = {J. Phys. Chem. C},
	number = {3},
	pages = {1989-1993},
	title = {Kinetics of {T}opological {S}tone--{W}ales {D}efect {F}ormation in {S}ingle-{W}alled {C}arbon {N}anotubes},
	volume = {120},
	year = {2016},
    doi = {10.1021/acs.jpcc.5b11682}}

@article{DNPV2012,
	author = {Di Nezza, Eleonora and Palatucci, Giampiero and Valdinoci, Enrico},
	doi = {10.1016/j.bulsci.2011.12.004},
	fjournal = {Bulletin des Sciences Math\'ematiques},
	issn = {0007-4497,1952-4773},
	journal = {Bull. Sci. Math.},
	mrclass = {46E35 (35A23 35S05 35S30)},
	mrnumber = {2944369},
	mrreviewer = {Lanzhe\ Liu},
	number = {5},
	pages = {521--573},
	title = {Hitchhiker's guide to the fractional {S}obolev spaces},
	volume = {136},
	year = {2012}}

@manual{petsc-tao-users-manual,
	author = {S. Balay and S. Abhyankar and others},
	institution = {Argonne National Laboratory},
	title = {{PETSc/TAO Users Manual}},
	year = {2024}}

@book{Z97,
	author = {Zubov, L. M.},
	publisher = {Springer},
	series = {Lecture Notes in Physics Monographs},
	title = {Nonlinear {T}heory of {D}islocations and {D}isclinations in {E}lastic {B}odies},
	year = {1997}}

@article{PolyharmonicGreenFunction,
	article-number = {543},
	author = {Karachik, V.},
	doi = {10.3390/axioms12060543},
	issn = {2075-1680},
	journal = {Axioms},
	title = {On {Green's} {F}unction of the {D}irichlet {P}roblem for the {P}olyharmonic {E}quation in the {B}all},
	volume = {12},
	year = {2023}}

@article{SN88,
	author = {Seung, H. S. and Nelson, D. R.},
	journal = {Phys. Rev. A},
	pages = {1005--1018},
	title = {Defects in flexible membranes with crystalline order},
	volume = {38},
	year = {1988},
    doi = {10.1103/PhysRevA.38.1005}}

@misc{BarattaEtal2023,
	author = {Baratta, I. A. and Dean, J. P. and others},
	doi = {10.5281/zenodo.10447666},
	howpublished = {preprint},
	title = {{DOLFINx}: the next generation {FEniCS} problem solving environment},
	year = {2023}}

@article{Brenner:Von_Karman,
	address = {Berlin, Heidelberg},
	author = {Brenner, S. C. and Neilan, M. and others},
	doi = {10.1007/s00211-016-0817-y},
	issn = {0029-599X},
	issue_date = {March 2017},
	journal = {Numer. Math.},
	numpages = {30},
	pages = {803--832},
	publisher = {Springer-Verlag},
	title = {A {$C^{0}$} interior penalty method for a {V}on {K{\'a}rm{\'a}n} plate},
	volume = {135},
	year = {2017}}

@article{CPL14,
	author = {Cesana, P. and Porta, M. and Lookman, T.},
	journal = {J. of the Mech. and Phys. of Solids},
	pages = {174--192},
	title = {Asymptotic analysis of hierarchical martensitic microstructure},
	volume = {72},
	year = {2014},
    DOI = {10.1016/j.jmps.2014.08.001}}

@article{GK2000,
	author = {Geymonat, Giuseppe and Krasucki, Fran{\c c}oise},
	doi = {10.1016/S0764-4442(00)00196-8},
	fjournal = {Comptes Rendus de l'Acad\'emie des Sciences. S\'erie I. Math\'ematique},
	issn = {0764-4442},
	journal = {C. R. Acad. Sci. Paris S\'er. I Math.},
	mrclass = {65N30 (35Q30)},
	mrnumber = {1751670},
	number = {5},
	pages = {355--360},
	title = {On the existence of the {A}iry function in {L}ipschitz domains. {A}pplication to the traces of {$H^2$}},
	volume = {330},
	year = {2000}}

@article{PETROLO20042471,
	author = {Antonio Salvatore Petrolo and Raffaele Casciaro},
	doi = {https://doi.org/10.1016/j.compstruc.2004.07.004},
	issn = {0045-7949},
	journal = {Computers \& Structures},
	number = {29},
	pages = {2471-2481},
	title = {3{D} beam element based on {S}aint {V}\'{e}nant's rod theory},
	volume = {82},
	year = {2004}}

@article{vG2017,
	author = {Van Goethem, Nicolas},
	doi = {10.1177/1081286516642817},
	journal = {Mathematics and Mechanics of Solids},
	number = {8},
	pages = {1688-1695},
	title = {Incompatibility-governed singularities in linear elasticity with dislocations},
	volume = {22},
	year = {2017}}

@article{vaGoethemDupret2012,
	author = {Van Goethem, Nicolas and Dupret, Fran\c{c}ois},
	doi = {10.1017/S0956792512000010},
	journal = {European Journal of Applied Mathematics},
	number = {3},
	pages = {417--439},
	title = {A distributional approach to 2D {V}olterra dislocations at the continuum scale},
	volume = {23},
	year = {2012}}

@article{Angoshtari2016,
	author = {Angoshtari, Arzhang and Yavari, Arash},
	doi = {10.1007/s00161-015-0478-6},
	journal = {Continuum Mechanics and Thermodynamics},
	number = {5},
	pages = {1347--1359},
	title = {The weak compatibility equations of nonlinear elasticity and the insufficiency of the {H}adamard jump condition for non-simply connected bodies},
	volume = {28},
	year = {2016}}

@inbook{Yavari2020,
	address = {Cham},
	author = {Yavari, Arash},
	booktitle = {Geometric Continuum Mechanics},
	doi = {10.1007/978-3-030-42683-5_3},
	editor = {Segev, Reuven and Epstein, Marcelo},
	isbn = {978-3-030-42683-5},
	pages = {143--183},
	publisher = {Springer International Publishing},
	title = {Applications of {A}lgebraic {T}opology in {E}lasticity},
	year = {2020}}

@article{ACGK,
	author = {C. Amrouche and P. G. Ciarlet and L. Gratie and S. Kesavan},
	journal = {C. R. Acad. Sci. Paris, Ser. I},
	pages = {887--891},
	title = {On {S}aint {V}enant's compatibility conditions and {P}oincar\'{e}'s lemma},
	volume = {342},
	year = {2006},
    DOI = {10.1016/j.crma.2006.03.026}}

@article{CARBONARA7,
	author = {Walter Lacarbonara and Achille Paolone},
	doi = {https://doi.org/10.1016/j.cam.2006.08.008},
	issn = {0377-0427},
	journal = {Journal of Computational and Applied Mathematics},
	number = {1},
	pages = {473-497},
	title = {On solution strategies to {S}aint-{V}enant problem},
	volume = {206},
	year = {2007}}

@article{Acharya99a,
	author = {Acharya, Amit},
	date = {1999/08/01},
	doi = {10.1023/A:1007653400249},
	id = {Acharya1999},
	isbn = {1573-2681},
	journal = {Journal of Elasticity},
	number = {2},
	pages = {95--105},
	title = {On {C}ompatibility {C}onditions for the {L}eft {C}auchy--{G}reen {D}eformation {F}ield in {T}hree {D}imensions},
	volume = {56},
	year = {1999}}

@article{Yavari2013,
	author = {Yavari, Arash},
	doi = {10.1007/s00205-013-0621-0},
	fjournal = {Archive for Rational Mechanics and Analysis},
	issn = {0003-9527,1432-0673},
	journal = {Arch. Ration. Mech. Anal.},
	mrclass = {74B05 (74B20)},
	mrnumber = {3054603},
	mrreviewer = {Ioan\ Bucataru},
	number = {1},
	pages = {237--253},
	title = {Compatibility equations of nonlinear elasticity for non-simply-connected bodies},
	volume = {209},
	year = {2013}}

@article{Michell,
	author = {Michell, J. H.},
	doi = {10.1112/plms/s1-31.1.100},
	issn = {0024-6115},
	journal = {Proceedings of the London Mathematical Society},
	month = {04},
	number = {1},
	pages = {100-124},
	title = {{On the Direct Determination of Stress in an Elastic Solid, with application to the Theory of Plates}},
	volume = {s1-31},
	year = {1899}}

@article{ZA2018,
	author = {C. Zhang and A. Acharya},
	doi = {https://doi.org/10.1016/j.jmps.2018.06.020},
	issn = {0022-5096},
	journal = {Journal of the Mechanics and Physics of Solids},
	pages = {188-223},
	title = {On the relevance of generalized disclinations in defect mechanics},
	volume = {119},
	year = {2018}}

@article{HAGIHARA10,
	author = {K. Hagihara and N. Yokotani and Y. Umakoshi},
	doi = {https://doi.org/10.1016/j.intermet.2009.07.014},
	issn = {0966-9795},
	journal = {Intermetallics},
	number = {2},
	pages = {267-276},
	title = {Plastic deformation behavior of {M}g12{YZ}n with 18{R} long-period stacking ordered structure},
	volume = {18},
	year = {2010}}

@article{LI72,
	author = {J. C. M. Li},
	doi = {https://doi.org/10.1016/0039-6028(72)90251-8},
	issn = {0039-6028},
	journal = {Surface Science},
	pages = {12-26},
	title = {Disclination model of high angle grain boundaries},
	volume = {31},
	year = {1972}}

@article{Gertsman89,
	author = {V. Yu. Gertsman and A. A. Nazarov and A. E. Romanov and R. Z. Valiev and V. I. Vladimirov},
	doi = {10.1080/01418618908209841},
	journal = {Philosophical Magazine A},
	number = {5},
	pages = {1113-1118},
	publisher = {Taylor & Francis},
	title = {Disclination-structural unit model of grain boundaries},
	volume = {59},
	year = {1989}}

@article{Eshelby66,
	author = {J. D. Eshelby},
	doi = {10.1088/0508-3443/17/9/303},
	journal = {British Journal of Applied Physics},
	month = {sep},
	number = {9},
	pages = {1131--1135},
	publisher = {{IOP} Publishing},
	title = {A simple derivation of the elastic field of an edge dislocation},
	volume = {17},
	year = 1966}

@book{ciarlet97,
	author = {P. G. Ciarlet},
	doi = {https://doi.org/10.1016/S0168-2024(97)80014-8},
	issn = {0168-2024},
	publisher = {Elsevier, North-Holland Publishing Co., Amsterdam},
	series = {Studies in Mathematics and Its Applications},
	title = {Mathematical Elasticity, Volume $\textrm{II}$: Theory of Plates},
	volume = {27},
	year = {1997}}

@inproceedings{acharya15,
	address = {Cham},
	author = {Acharya, A. and Fressengeas, C.},
	booktitle = {Differential Geometry and Continuum Mechanics},
	editor = {Chen, Gui-Qiang G. and Grinfeld, Michael and Knops, R. J.},
	isbn = {978-3-319-18573-6},
	pages = {123--165},
	publisher = {Springer International Publishing},
	title = {Continuum Mechanics of the Interaction of Phase Boundaries and Dislocations in Solids},
	year = {2015},
    DOI = {10.1007/978-3-319-18573-6\_5}}

@incollection{RV92,
	address = {Amsterdam},
	author = {Romanov, A.E. and Vladimirov, V.I.},
	booktitle = {Dislocations in solids},
	editor = {Nabarro, F.R.N.},
	pages = {191},
	publisher = {North-Holland},
	title = {Disclinations in crystalline solids},
	volume = {9},
	year = {1992}}

@article{I19,
	author = {Inamura, T.},
	journal = {Acta Materialia},
	pages = {270-280},
	title = {Geometry of kink microstructure analysed by rank-1 connection},
	volume = {173},
	year = {2019},
    doi = {10.1016/j.actamat.2019.05.023}}

@book{B,
	author = {Bhattacharya, K.},
	publisher = {Oxford University Press},
	title = {Microstructure of martensite: why it forms and how it gives rise to the shape-memory effect},
	year = {2003}}

@article{V07,
	author = {Volterra, V.},
	journal = {Annales scientifiques de l'{\'E}cole Normale Sup{\'e}rieure},
	pages = {401-517},
	title = {Sur l'{\'e}quilibre des corps {\'e}lastiques multiplement connexes},
	volume = {24},
	year = {1907}}

@article{CermelliLeoni06,
	author = {Cermelli, P. and Leoni, G.},
	journal = {SIAM Journal on Mathematical Analysis},
	number = {4},
	pages = {1131--1160},
	publisher = {[Philadelphia] Society for Industrial and Applied Mathematics.},
	title = {Renormalized energy and forces on dislocations},
	volume = {37},
	year = {2005},
    DOI = {10.1137/040621636}}

@article{DeLucaGarroniPonsiglione12,
	author = {De Luca, L. and Garroni, A. and Ponsiglione, M.},
	journal = {Archive for Rational Mechanics and Analysis},
	pages = {885--910},
	title = {{$\Gamma$}-convergence analysis of systems of edge dislocations: the self energy regime},
	volume = {206(3)},
	year = {2012},
    DOI = {10.1007/s00205-012-0546-z}}

@article{FleckMullerAshbyHutchinson94,
	author = {Fleck, N. A. and Muller, G. M. and Ashby, M. F. and Hutchinson, J. W.},
	journal = {Acta Metallurgica et Materialia},
	pages = {475--487},
	title = {Strain gradient plasticity: theory and experiment},
	volume = {42(2)},
	year = {1994},
    doi = {10.1016/0956-7151(94)90502-9}}

@article{GarroniLeoniPonsiglione10,
	author = {Garroni, A. and Leoni, G. and Ponsiglione, M.},
	journal = {Journal European Mathematical Society},
	pages = {1231--1266},
	title = {Gradient theory for plasticity via homogenization of discrete dislocations},
	volume = {12(5)},
	year = {2010},
    DOI = {10.4171/JEMS/228}}

@article{Peierls40,
	author = {Peierls, R.},
	journal = {Proceedings of the Physical Society},
	number = {1},
	pages = {34--37},
	publisher = {World Scientific},
	title = {The size of a dislocation},
	volume = {52},
	year = {1940}}

@article{Ginster19_2,
	author = {Ginster, J.},
	journal = {SIAM J. Math. Anal.},
	pages = {3424--3464},
	publisher = {SIAM},
	title = {Strain-gradient plasticity as the {$\Gamma$}-limit of a nonlinear dislocation energy with mixed growth},
	volume = {51},
	year = {2019},
    DOI = {10.1137/18M1176579}}

@article{BlassMorandotti17,
	author = {Blass, T. and Morandotti, M.},
	journal = {Journal of Convex Analysis},
	number = {2},
	pages = {547--570},
	title = {Renormalized energy and {P}each-{K}{\"o}hler forces for screw dislocations with antiplane shear},
	volume = {24},
	year = {2017}}

@article{Georgoulis08,
	author = {Georgoulis, Emmanuil H. and Houston, Paul},
	doi = {10.1093/imanum/drn015},
	issn = {0272-4979},
	journal = {IMA Journal of Numerical Analysis},
	month = {07},
	number = {3},
	pages = {573-594},
	title = {Discontinuous {G}alerkin methods for the biharmonic problem},
	volume = {29},
	year = {2008}}

@inbook{Brenner2012,
	address = {Berlin, Heidelberg},
	author = {Brenner, Susanne C.},
	booktitle = {Frontiers in Numerical Analysis - Durham 2010},
	doi = {10.1007/978-3-642-23914-4_2},
	editor = {Blowey, James and Jensen, Max},
	isbn = {978-3-642-23914-4},
	pages = {79--147},
	publisher = {Springer Berlin Heidelberg},
	title = {$C^0$ {I}nterior {P}enalty {M}ethods},
	year = {2012}}

@article{Arnold2002,
	author = {Arnold, Douglas N. and Brezzi, Franco and Cockburn, Bernardo and Marini, L. Donatella},
	doi = {10.1137/S0036142901384162},
	journal = {SIAM Journal on Numerical Analysis},
	number = {5},
	pages = {1749-1779},
	title = {Unified {A}nalysis of {D}iscontinuous {G}alerkin {M}ethods for {E}lliptic {P}roblems},
	volume = {39},
	year = {2002}}

@article{Arnold82,
	author = {Arnold, Douglas N.},
	doi = {10.1137/0719052},
	journal = {SIAM Journal on Numerical Analysis},
	number = {4},
	pages = {742-760},
	title = {An {I}nterior {P}enalty {F}inite {E}lement {M}ethod with {D}iscontinuous {E}lements},
	volume = {19},
	year = {1982}}

@article{CFM2025,
	author = {Pierluigi Cesana and Edoardo Fabbrini and Marco Morandotti},
	doi = {10.1007/s10659-025-10161-5},
	issn = {1573-2681},
	journal = {Journal of Elasticity},
	number = {4},
	pages = {71},
	title = {Variational Formulation of {P}lanar {L}inearized {E}lasticity with {I}ncompatible {K}inematics},
	volume = {157},
	year = {2025}}

@article{AINSWORTH2024,
	author = {Mark Ainsworth and Charles Parker},
	doi = {https://doi.org/10.1016/j.cma.2024.117267},
	issn = {0045-7825},
	journal = {Computer Methods in Applied Mechanics and Engineering},
	pages = {117267},
	title = {Two and three dimensional ${H}^2$-conforming finite element approximations without ${C}^1$-elements},
	volume = {431},
	year = {2024}}

@article{FARRELL2022,
	author = {Patrick E. Farrell and Abdalaziz Hamdan and Scott P. MacLachlan},
	doi = {https://doi.org/10.1016/j.camwa.2022.10.024},
	issn = {0898-1221},
	journal = {Computers \& Mathematics with Applications},
	pages = {300-319},
	title = {A new mixed finite-element method for {$H^2$} elliptic problems},
	volume = {128},
	year = {2022}}

@article{BALASUNDARAM1984,
	author = {S. Balasundaram and P.K. Bhattacharyya},
	doi = {https://doi.org/10.1016/0898-1221(84)90052-X},
	issn = {0898-1221},
	journal = {Computers \& Mathematics with Applications},
	number = {3},
	pages = {245-256},
	title = {A mixed finite element method for fourth order elliptic equations with variable coefficients},
	volume = {10},
	year = {1984}}

@article{PAPANICOLOPULOS2013,
	author = {S.-A. Papanicolopulos and A. Zervos},
	doi = {https://doi.org/10.1016/j.compstruc.2012.07.003},
	issn = {0045-7949},
	journal = {Computers \& Structures},
	note = {Special Issue: UK Association for Computational Mechanics in Engineering},
	pages = {53-58},
	title = {Polynomial ${C}^1$ shape functions on the triangle},
	volume = {118},
	year = {2013}}

@article{PAPANICOLOPULOS2012,
	author = {Papanicolopulos, S.-A. and Zervos, A.},
	doi = {https://doi.org/10.1002/nme.3296},
	journal = {International Journal for Numerical Methods in Engineering},
	number = {11},
	pages = {1437-1450},
	title = {A method for creating a class of triangular ${C}^1$ finite elements},
	volume = {89},
	year = {2012}}
\end{document}